\documentclass[12pt]{amsart}
\usepackage{color,graphicx,array, amssymb, amscd,slashed}
\usepackage[all]{xy}
\usepackage{tikz-cd}
\newtheorem{Theorem}{Theorem}[section]
\newtheorem{Lemma}[Theorem]{Lemma}
\newtheorem{Proposition}[Theorem]{Proposition}
\newtheorem{Corollary}[Theorem]{Corollary}
\theoremstyle{definition}
\newtheorem{Definition}{Definition}
\theoremstyle{remark}
\newtheorem{Example}{Example}
\newtheorem{Remark}[Theorem]{Remark} 

\numberwithin{equation}{section}
\newcommand{\R}{\mathbb R}

\newcommand{\C}{\mathbb C}

\newcommand{\D}{\mathbb D}

\newcommand{\St}{\mathbb S}

\newcommand{\U}{{\rm U}_{2,1}}

\newcommand{\SOt}{{\rm SO}_{2,1}}
\newcommand{\SU}{{\rm SU}_{2, 1}}

\newcommand{\stabU}{{\rm U}_{1, 1} \times{\rm U}_1 }

\newcommand{\Uoneone}{{\rm U}_{1, 1}}
\newcommand{\SUoneone}{{\rm SU}_{1, 1}}
\newcommand{\id}{\operatorname{id}}

\newcommand{\Gr}{\operatorname{Gr}}
\newcommand{\LGr}{\operatorname{L_{Gr}}}
\newcommand{\SLGr}{\operatorname{SL_{Gr}}}

\newcommand{\SL}{{\rm SL}_3 \mathbb C}

\renewcommand{\sl}{\mathfrak{sl}_3 \mathbb C}

\newcommand{\lsl}{\Lambda \mathfrak{sl}_3 \mathbb C_{\sigma}}
\newcommand{\LSL}{\Lambda {\rm SL}_3 \mathbb C_{\sigma}}

\newcommand{\ad}{\operatorname{Ad}}
\newcommand{\di}{\operatorname{diag}}
\newcommand{\tr}{\operatorname{Tr}}

\renewcommand{\Re}{\operatorname {Re}}
\renewcommand{\Im}{\operatorname {Im}}
\newcommand{\CH}{\mathbb {CH}^{2}_1}
\newcommand{\CP}{\mathbb {CP}^{2}_1}
\newcommand{\f}{\mathfrak f}
\renewcommand{\d}{\mathrm d}

\renewcommand{\r}[1]{ #1_0}

\renewcommand{\l}{\lambda}

\usepackage{amsmath}	
\begin{document}
\title{Timelike minimal Lagrangian surfaces in the indefinite 
 complex hyperbolic two-space}
 \author[J. F.~Dorfmeister]{Josef F. Dorfmeister}
 \address{Fakult\"at f\"ur Mathematik, 
 TU-M\"unchen, 
 Boltzmann str. 3,
 D-85747, 
 Garching, 
 Germany}
 \email{dorfm@ma.tum.de}
 \author[S.-P.~Kobayashi]{Shimpei Kobayashi}
 \address{Department of Mathematics, Hokkaido University, 
 Sapporo, 060-0810, Japan}
 \email{shimpei@math.sci.hokudai.ac.jp}
 \thanks{The second named author is partially supported by JSPS 
 KAKENHI Grant Number JP18K03265 and 
 Deutsche Forschungsgemeinschaft-Collaborative Research Center, TRR 109, ``Discretization in Geometry and Dynamics''.}
 \date{\today}
\pagestyle{plain}
\begin{abstract}
 It has been known for some time that there exist $5$ essentially different
 real forms of the complex affine Kac-Moody  algebra of type $A_2^{(2)}$ 
 and that one can associate $4$ of these real forms with certain classes 
 of ``integrable surfaces'', such as
  minimal Lagrangian surfaces in $\mathbb {CP}^2$ and 
 $\mathbb {CH}^2$, as well as definite and indefinite 
 affine spheres in $\R^3$.
 
 In this paper we consider the  class of 
 timelike minimal Lagrangian surfaces in the indefinite complex 
 hyperbolic two-space $\CH$. We show that this class of surfaces corresponds
 to the fifth real form.

 Moreover, for each  timelike Lagrangian surface in $\CH$ 
 we define  natural Gauss maps into certain homogeneous spaces and 
 prove a Ruh-Vilms type theorem, characterizing  timelike minimal 
 Lagrangian surfaces among all timelike Lagrangian surfaces 
 in terms of the harmonicity of these Gauss maps.
\end{abstract}
\maketitle
\section*{Introduction}
 It became more and more clear in recent years that many surface classes are characterized by harmonic maps into some $k$-symmetric space.
 In the classical case Ruh-Vilms \cite{RV} 
 have characterized all constant mean curvature surfaces
 in $\R^3$ among all surfaces as those for which the (classical) Gauss map into the symmetric space $\St^2 = \textrm{SO}_3/\textrm{SO}_2$ 
 is harmonic.
 Another case consists of all constant mean curvature surfaces in the real hyperbolic 
 space $\mathbb{H}^3$, which are those surfaces in $\mathbb{H}^3$ for which the 
 ``normal Gauss map'' into the unit tangent bundle of $\mathbb{H}^3$, considered as a $4$-symmetric space, is harmonic, \cite{DIK}.
 Another group of surfaces with an analogous characterization seem to be the definite and the indefinite affine spheres, and the minimal Lagrangian surfaces 
 in $\mathbb{CP}^2$  and in $\C \mathbb{H}^2$, see for examples, 
 \cite{DE, DW1, DM, ML}. 
 (So far only in \cite{McIn} a harmonic ``Gauss map'' is given explicitly.)

 All these examples come in  $S^1$-families of surfaces of the same class and can be investigated by using the loop group technique. Here one observes that the naturally associated moving frames of an associated family are contained in a specific loop group.
 In \cite{DE} it was observed that the indefinite affine spheres in $\R^3$ are associated with
 a real form of the affine Kac-Moody algebra of type $A_2^{(2)}$. Later it was observed that the definite affine spheres (of elliptic type or of hyperbolic type) also are associated with a real form of type $A_2^{(2)}$, as well as the minimal Lagrangian immersions into $\mathbb{CP}^2$ and the 
 the minimal Lagrangian immersions into $\mathbb{CH}^2$. In view of the classification of 
 all real forms of the complex affine Kac-Moody algebra of type $A_2^{(2)}$ by Heintze-Gro{\ss} \cite{HG} or Rousseau et al. \cite{B3R, BMR} it became clear, that the surface types mentioned above correspond exactly to four of the five types of inequivalent real forms, see \cite{DFKW}. 
 For the case of the complex affine Kac-Moody algebra of type 
$A_1^{(1)}$ it has been shown in \cite{Ko;real} that the real forms of this Kac-Moody algebra are related to constant mean curvature/constant Gaussian curvature 
surfaces in the Euclidean $3$-space, the Minkowski $3$-space
 or in the hyperbolic $3$-space.

 In this paper we present the ``missing case''. More precisely, we define 
 \emph{timelike} minimal Lagrangian surfaces in the indefinite complex hyperbolic space which are associated with the missing real form, and also define a Gauss map for all Lagrangian surfaces in the indefinite complex hyperbolic space. These Gauss maps take values in a 
\textit{quasi} $6$-symmetric space (see Definition \ref{def:k-symmetric})
 and are Lorentz primitive harmonic if and only if the corresponding Lagrangian surfaces in the indefinite complex hyperbolic space are minimal.
 We note that in \cite{DIT, IT}, the loop group methods for 
 timelike constant mean 
 curvature surfaces and timelike minimal surfaces in the Minkowski $3$-space 
$\R^3_1$ have been developed, and these surfaces
 correspond to a real form of the complex 
 affine Kac-Moody algebra of type $A_1^{(1)}$.

This result permits to apply the loop group technique which represents a general procedure to construct all surfaces of the associated class; in our case all minimal timelike Lagrangian surfaces in the indefinite complex hyperbolic space.
More on this is left to a separate investigation.

\section{Timelike minimal Lagrangian surfaces in $\CH$}
 In this section, we define  timelike Lagrangian surfaces in $\CH$ and
 discuss their basic properties.
 In particular 
 we characterize minimality of a timelike Lagrangian surface 
 by the vanishing of the so-called ``mean curvature'' 
 $1$-form, Proposition \ref{prp:minicharact}.
\subsection{Surfaces in $\CH$}
 Let
\begin{equation}\label{eq:P}
 \r{P} = 
\begin{pmatrix}
 0 & 1 & 0 \\ 
 1 & 0 & 0 \\ 
 0 & 0 & -1 
\end{pmatrix},
\end{equation}
 and consider the three-dimensional complex Hermitian flat space $\C^3_2$, that is, $\C^3$ together with the pseudo-Hermitian form of signature 
 $(2,1)$
\begin{equation} \label{Hermitian}
 \langle z, w\rangle = z^T \r{P} \,  \!\bar w = 
 z_1 \overline{w_2} + z_2 \overline{w_1} -z_3 \overline{w_3}.
\end{equation}
 Vectors $v\in \C^3_2$ satisfying $\langle v,v \rangle < 0$ or $\langle v,v \rangle > 0$ will be called ``negative'' and ``positive'' respectively. Clearly, the set of these vectors is open in 
 $\C^3_2$ and $\C^{\times}$ acts freely on these sets by multiplication.
\begin{Definition}
 The real part and the imaginary part of the indefinite Hermitian 
 inner product of $\C^{3}_2$ define a pseudo-Riemannian metric $g$ 
 and a symplectic form $\Omega$, respectively:
 \begin{equation} \label{hermitianmetricphi}
  \langle \>,\> \rangle  = \Re\langle \>,\> \rangle + \sqrt{-1} 
 \Im\langle \>,\> \rangle = g (\>,\> ) + \sqrt{-1} \Omega (\>,\> ). 
 \end{equation}
\end{Definition}
 Then the indefinite complex hyperbolic space $\CH$, see \cite[Section 2]{DF}, defined by
 \begin{equation}
 \CH = \{ \mathbb C^{\times} v \;|\; v \in \C^3_2, \; \langle v,v \rangle < 0\}
 \end{equation}
 is a two-dimensional complex manifold.
 Let
 \[
 \U = \left\{A \;\Big|\; \begin{array}{l}
 \mbox{Invertible real linear map  in $\C^3_2$ 
 }  \\
 \mbox{satisfying  $\langle Av, Aw \rangle = \langle v, w\rangle$ 
 for all $v,w \in \C^3_2$.}
 \end{array}
 \right\}.
  \]
  Then $\C^{\times }\cdot \U$ is a connected reductive Lie group which acts transitively on the set of negative (resp. positive) 
 vectors. 
 As a consequence, $\U$ acts transitively on $\CH$ and it is easy 
 to verify that the stabilizer  in $\U$ of the negative vector 
 $e_3 = (0,0,1)^T$ is given by the
 diagonal block form matrix group  $\stabU$ in $\U$, where  
 $\Uoneone$, is the group of isometries of  
 the indefinite Hermitian metric  of $\C^2_1$ given by  
 $(z,w) = z_1 \overline{w_2} + z_2 \overline{w_1}$.
  
  As a consequence, $\CH$ can be 
 represented as the indefinite Hermitian symmetric space, see for example  \cite[Section 2]{Shapiro}:
 \begin{equation}
 \CH = \U/\stabU.
 \end{equation}
 The complex  manifold $\CH$ carries naturally the pseudo-Hermitian metric induced from $\C^3_2$. The projection is a pseudo-Riemannian submersion.
\begin{Remark}
 The indefinite complex hyperbolic space $\CH$ is known to be 
 anti-isometric (the metrics differ by a  minus sign) to the complex de Sitter space $\CP$ of
 all positive lines of complex Hermitian flat space $\C^3_1$ 
 with signature $(1,2)$, see \cite[p. 96]{A}.
\end{Remark}
 Let $H^5_3$ be the anti-de Sitter sphere 
 (note again that the signature of $\C^3_2$ is $(2, 1)$):
\[
 H^5_{3} = \left\{ v \in \C^3_2 \;|\;  \langle v, v\rangle = -1\right\}.
\]
 Then there exists the Boothby-Wang type fibration 
 $\pi : H^5_3 \to \CH$ given by $v \mapsto 
 \C^{\times} v$, \cite{BW, DF}.
 The tangent space of $H^5_3$ at $p \in H^5_3$ is
\begin{equation*}
 T_p H^5_3  = \{ w \in \C^3_2 \;|\;\Re \langle w,p \rangle =0\}.
\end{equation*}
 Moreover, the space 
 \begin{equation*}
 \mathcal{H}_p = \{ w \in T_p H^5_3\;|\;  \langle w,p \rangle =0\}
 \end{equation*}
 is a natural horizontal subspace.
 Recall that the projection $\pi$ from $H^5_3$  
 to $\CH$ is a pseudo-Riemannian submersion.
 Moreover, note that the form 
\[
\zeta (p) = \Im \langle p, \cdot  \rangle			      
\] 
 is a contact form and  $H^5_{3}$ is a contact manifold.
 Note also that $H^5_3$ can be represented as the symmetric space 
 \[
  H^5_3 = \U/\Uoneone,
 \]
 where $\Uoneone$ here more precisely means the block form matrix group 
 $\Uoneone \times \{1\}$.

 Since $\pi$ is a 
 pseudo-Riemannian submersion, we will make use of the pseudo-Riemannian metric
 $g$ and the symplectic 
 form $\Omega$ on $\CH$ which is given by 
\begin{equation}
g(a, b ) = \Re \langle \tilde a, \tilde b \rangle,  \quad 
\Omega(a,b) =  \Im \langle \tilde a, \tilde b \rangle, 
\end{equation}
 where $a, b \in T_p \CH$  and $\tilde a, \tilde b \in T_{\tilde p}H^5_3$ 
 are the vectors in the horizontal subspace 
 $\mathcal{H}_{\tilde p} \subset T_{\tilde p} H^5_3$ 
 corresponding uniquely to $a$ and $b$ respectively via $\pi.$
\begin{Lemma}
 Let $\D$ be a simply connected domain in $\R^2$ and 
 $f: \D \to \CH$ a Lagrangian map, $($thus satisfying $\Omega(\d f,\d f)=0$$)$.
 Then there exists a lift $\f : \D \to H^5_3$ such that 
\begin{equation}\label{eq:horizontal}
 \langle \d \f, \f \rangle = 0.
\end{equation}
 This lift is unique up to a constant factor from $S^1$.
 A lift $\f$ of a Lagrangian map $f$ 
 with the condition \eqref{eq:horizontal} as above
 will be called a {\rm horizontal lift}.
\end{Lemma}
\begin{proof}
 Let $\hat \f : \D \to H^5_3$ be a lift of $f$. Then $\langle \d \hat \f,  
 \hat \f\rangle$ + $\langle \hat \f,  
 \d \hat \f\rangle = 0$, that is, $\langle \d \hat \f,  
 \hat \f\rangle$ takes purely imaginary values. Moreover, 
 the Lagrangian condition for $\f$ means that 
 $\langle \d \hat \f,  \hat \f\rangle$ is a closed $1$-form.
 Since $\D$ is a simply connected domain in $\R^2$, the form $ \langle \d\hat \f, \hat \f \rangle$ is exact. Hence  there exists 
 a real function $\eta : \D \to \R$ such that $\sqrt{-1} \d \eta = \langle \d \hat \f, \hat \f \rangle$. Then we put  $\f =e^{\sqrt{-1} \eta} \hat \f$ and 
 $\langle \d \f, \f\rangle =0$ follows.
\end{proof}

\begin{Remark}
 A  horizontal lift $\f :\D \rightarrow H^5_3$  of $f$ is sometimes called 
 a  \textit{Legendre lift} of $f :\D \rightarrow \CH,$
 since for a horizontal lift of a Lagrangian immersion
 the condition \eqref{eq:horizontal} is equivalent with 
 $\zeta (\f(q)) (\d\f(q)) = 0$, equivalently  $\Im \langle \d \f, \f\rangle=0$, 
 and this means that $\f$ is 
 a Legendre immersion into the contact manifold $H^5_3$.
 For a more general discussion  of  the notion a Legendre lift see Section 
 \ref{sc:Lift}. 
\end{Remark}
 Let $f: M \to \CH$  be a Lagrangian immersion from a two-dimensional manifold $M$.
 Then $f$  induces a pseudo-Riemannian metric on $M$.
 If we restrict the immersion $f$ to any contractible open subset $\D$ of $M$, 
 then the induced metric of $f$ is represented, on $\D$,  by using the horizontal 
 lift $\f$,  as 
 \begin{equation} \label{metricf}
  \d s^2  =  \Re \langle \d\f, \d\f \rangle = g(\d f, \d f).
 \end{equation}
Note, the second equality above comes from the fact that 
 two horizontal lifts of $f$ 
 only differ by a constant scalar factor from $S^1$.

 In what  follows we will consider exclusively  timelike  surfaces. 
 Hence  the induced metric $\d s^2$ is assumed to be
 indefinite. Moreover, we always assume that all surfaces are Lagrangian.
 
\begin{Remark}
 For a Lagrangian immersion $f$ in $\CH$, 
 we obtain for the complex structure $J$ of $\CH$ 
 the identity $g(J\circ \d f, J \circ \d f ) = g(\d f, \d f) $.
 The definition of a Lagrangian surface implies that  
 $J \circ \d f$ is perpendicular to $\d f$ and a timelike vector.
 As a consequence, $g(\d f, \d f)$ is not  spacelike. Hence
 we have the following corollary.
 \begin{Corollary}
 There does not exist any spacelike Lagrangian surface in $\CH$.
 \end{Corollary}
\end{Remark}


\subsection{Moving frame}
 In this subsection we discuss moving frames of timelike Lagrangian surfaces in $\CH$. 
 First we discuss null coordinates, 
 that is, for an indefinite metric $g = \sum_{i, j} g_{ij}\, 
 \d x_i \d x_j$ 
 on a surface $M$, such that 
 $g_{11} = g_{22}=0$ and $g_{12} = g_{21} \neq 0$.
 The existence of null coordinates can be found for example in 
 \cite{TWeinstein} or \cite[Prop 14.1.18 and Remark 14.1.19]{B-Wood}. In our case, this result is formulated as follows:
\begin{Theorem}
 Let $f : \D \rightarrow \CH$ be a timelike Lagrangian immersion and 
 $\f$ a horizontal lift of $\f$. Then the metric  \eqref{metricf} 
 induced  by $\f$ $($and $f)$ on $\D$ is Lorentzian.
 In particular, in a neighbourhood of any point of $\D$
 {\it null coordinates}  exist for $\f$ $($and $f)$.
\end{Theorem}
 For a   horizontal lift $\f$ of a timelike Lagrangian immersion 
 $f$ we thus have:
 \begin{equation}\label{eq:nullf}
 \Re{\langle \f_u, \f_u\rangle}  = 
 \Re{\langle \f_v, \f_v\rangle}  = 0, \quad \mbox{and} \quad
 \mbox{$\Re{\langle \f_u, \f_v\rangle}$ never vanishes.}
\end{equation}
 From the horizontality $\langle \d \f, \f\rangle=0$, we have 
$\langle \f_u, \f \rangle= \langle \f_v, \f \rangle = 0$. 
 Moreover, taking the derivative with 
 respect to $v$ and $u$, respectively, we obtain  
 $\langle \f_u, \f_v \rangle
 =
 \langle \f_v, \f_u \rangle$. Hence,  $\langle \f_u, \f_v \rangle$ is real and never vanishes. Assuming without loss of generality 
 $\langle \f_v, \f_u \rangle >0$
  we finally obtain:
 \begin{equation}\label{eq:null}
 {\langle \f_u, \f_u\rangle}  = 
 {\langle \f_v, \f_v\rangle}  = 0, \quad \mbox{and} \quad
 \mbox{${\langle \f_u, \f_v\rangle} = \langle \f_v, \f_u \rangle$ is always positive,}
\end{equation}
and
\[
  \Im \langle \f_u, \f_v \rangle  = \Omega (f_u, f_v) =0.
\]
 Moreover, we have just seen that $\langle \f_u, \f_v \rangle$ is always
 positive. Therefore, we can assume that there exists a
 real function 
 $\omega : \D \to \R$ such that 
 \[
  \langle \f_u, \f_v \rangle = e^{\omega} \quad \mbox{and}
\quad \d s^2 = 2 e^{\omega} \d u \d v
 \]
 holds. Then we consider the  \textit{coordinate frame}
 \begin{equation}\label{eq:coordinateframe}
 \mathcal F = \left( e^{-\omega/2}\f_u, e^{-\omega/2} \f_v, \f\right).
\end{equation}
 It is straightforward to see that $\mathcal F$ takes values in 
 $\U$, that is, 
\begin{equation}\label{eq:SU21}
 \r{P} \mathcal F^T \r{P}\, \bar {\mathcal F}  = 
 \id, \quad
 \mbox{equivalently \; $\mathcal F^{-1} = \r{P} \bar{\mathcal F}^T \r{P}$}
\end{equation} 
 holds, where $\r{P}$ is defined in \eqref{eq:P}. 
 Then $|\det \mathcal{F}|^2 = 1$ and 
 $\mathcal{F} \in \U = S^1 \cdot \SU $ follows.
  
 We now want to compute the Maurer-Cartan form of $\mathcal{F}$. For this we will use the mean curvature vector of $f$. 
 
 First we consider the decomposition
 \[
  \D \times \C^3_{2} = \d\f (\R^2) \oplus \d\f(\R^2)^\perp  \oplus \d\f(\R^2)^{\perp \perp}
 \] 
of the trivial bundle into three real, pairwise perpendicular rank 
 $2$ subbundles, where 
\[
 \mathcal{H} = \d\f (\R^2) \oplus \d\f(\R^2)^\perp
\]
 is the natural horizontal subspace of $T_{\f(z)} H^5_3 .$
 The vectors $\{\f_u, \f_v\}$, $\{\sqrt{-1}\f_u, \sqrt{-1}\f_v\}$, and $\{\sqrt{-1}\f, \f\}$ form 
 a basis of these real two-dimensional subspaces respectively. 
 Their pairwise Hermitian products can be read off from the formulas listed just above.
 By the definition of the metric $g$ the projection $\hat \pi$ 
 from $\D \times \C^3_{2} $ to $\C^3_2$ induces 
 an isometry from 
 $\d\f (\R^2)$ onto  $\d f(\R^2)$ and from $\d \f (\R^2)^\perp$ onto $\d f(\R^2)^\perp $.
 The vector $\sqrt{-1}\f$ is annihilated by $\d \pi$.

 More precisely, putting $E_1 = \f_u$ and $E_2 = \f_v$ we obtain  basis vectors of 
  $\d \f(\R^2)$. 
 Then $\sqrt{-1}E_1= \sqrt{-1}\f_u$ and  $\sqrt{-1}E_2= \sqrt{-1}\f_v $ and $\sqrt{-1}\f_u$ and 
 $\sqrt{-1}\f_v$ are 
 in $\d \f(\R^2)^\perp$. The differential of $\pi$ maps the vectors onto 
  $\d f(\R^2)$ and $\d f(\R^2)^\perp$, respectively.

 Recall that from \cite[Theorem 1, c)]{Reck}, the second fundamental form 
 $I\!I^{f}$ for a timelike Lagrangian surface $f$ in $\CH$
 can be obtained by the second fundamental form $I\!I^{\f}$ for a
 horizontal lift $\f$ in $H^5_3$ which takes values in 
 the horizontal subspace $\mathcal H$, given as $I\!I^{f} = \d \pi I\!I^{\f}$:
 see \cite[Theorem 1]{Reck}.
 Since $\f$ is in $H^5_3 \subset \C^3_2$, and  $XY \f$ takes values in $\C^3_2$ and 
 the second fundamental form $I\!I^{\f}$ can be given by 
\[
  I\!I^{\f} (X, Y ) = g(XY \f, e_1)e_1 - g(XY \f, e_2 )e_2,
 \]
 with $X, Y \in \Gamma (T \D)$ and  $e_1$ and $e_2$ being perpendicular 
 vectors of $\d \f(\R^2)^\perp$ of ``length'' 
 $1$ and $-1$ respectively. Note that $ I\!I^{\f} (X, Y ) $ takes values in $T_p \mathcal H$.
 Then the mean curvature vector $\mathfrak H$ of $\f$
 is defined by $\frac{1}{2} \mathrm{Tr}_{g} I\!I^{\f}$, 
 that is, 
\begin{equation} \label{Horig}
 \mathfrak H = \frac{1}{2} \left\{I\!I^{\f}(\partial_s, \partial_s) -I\!I^{\f}(\partial_t, \partial_t)\right\}
 =\frac{1}{2} \left\{g(\partial_s^2 \f - \partial_t^2 \f,  e_1 )e_1 - g(\partial_s^2 \f-\partial_t^2 \f, e_2 )e_2 \right\},
\end{equation}
 where $\{\partial_s, \partial_t\}$ is the orthonormal frame with respect to 
 the indefinite metric $\d s^2$.
 In our case we define $e_1 = \frac{\sqrt{-1}}{\sqrt{2}}(E_1 + E_2) e^{-\frac{\omega}{2}}$
 and $e_2 = \frac{\sqrt{-1}}{\sqrt{2}}(E_1 - E_2) e^{-\frac{\omega}{2}}$ and 
 $\partial_s = \frac{1}{\sqrt{2}}e^{-\frac{\omega}{2}}(\partial_u + \partial_v)$, 
 $\partial_t = \frac{1}{\sqrt{2}}e^{-\frac{\omega}{2}}(\partial_u - \partial_v).$
Then, from equation \eqref{Horig} it follows by a straightforward computation
\begin{equation*} 
\mathfrak H = e^{-2 \omega} \left\{g(\f_{uv}, \sqrt{-1}E_1 )\sqrt{-1}E_2 +  g(\f_{uv}, \sqrt{-1}E_2 )\sqrt{-1}E_1\right\}.
\end{equation*}
 Now the mean curvature $H$ of the original immersion $f$ is
 given by 
\begin{equation*}
 \d\pi (\mathfrak{H}) = H.
\end{equation*}

 By abuse of notation we will also call $\mathfrak{H}$
 the \textit{mean curvature vector} of $f$.
 A straightforward computation shows that we obtain the following 
 description of $\mathfrak{H}$:
\begin{equation*}
\mathfrak{H}=  g(e^{- \omega}\f_{uv}, \mathfrak E_1 )\mathfrak E_2  
 +  g(e^{- \omega}\f_{uv}, \mathfrak E_2   )\mathfrak E_1,
\end{equation*}
 where $\{\mathfrak E_1, \mathfrak E_2\}=\{\sqrt{-1}e^{-\omega/2}\f_u, \sqrt{-1}e^{-\omega/2}\f_v\}$
 is a null basis of $\d\f(\R^2)^\perp$.
 We thus  compute $\mathfrak{H}$ as the component of $e^{-\omega}  \f_{uv}$ 
 in $\d\f (\R^2)^\perp$. 
 Since  $g(\f_{uv}, \f_u) = g(\f_{uv}, \f_v)=0,$ we obtain 
$e^{- \omega}\f_{uv} = \mathfrak H + a i\f + b \f.$ Taking inner products yields $b = - g(\f_{uv},\f) = 1$ and
$a=0$. Thus altogether we obtain:
\begin{equation}\label{eq:Hexpression}
\mathfrak H =e^{-\omega}  \f_{uv} - \f.
\end{equation}

\begin{Remark}
 In general, for a surface $\f$ in $H^5_3$, the second fundamental form 
 can be written in the form 
 \[
  I\!I(X, Y ) = g(XY \f, \mathfrak{e}_1) \mathfrak{e}_1 - 
  g(XY \f,  \mathfrak{e}_2 ) \mathfrak{e}_2 
 -g(XY \f,  \mathfrak{e}_3 ) \mathfrak{e}_3,
 \]
 where $\{\mathfrak{e_1},\mathfrak{e_2}, \mathfrak e_3 \}$ are perpendicular vectors 
 in $\d \f (\R^2)^{\perp} \subset T_{\f(p)} H^5_3$ of lengths $1$ and $-1$ and
 $-1$, respectively.
 If $\f$ is a Legendre immersion, then $\mathfrak{e}_3 = \sqrt{-1} \f$
 and $g(XY \f, \mathfrak{e}_3 )=0$ from the Legendrian condition 
 $\zeta (\f(q)) (\d \f (q))=0$.
\end{Remark}

 Now we can prove the following theorem.
\begin{Theorem}
 The Maurer-Cartan form 
 $\mathcal F^{-1} \d \mathcal F = \mathcal F^{-1} \mathcal F_u \d u  
 +  \mathcal F^{-1} \mathcal F_v \d v$
 can be computed as 
\begin{align}
 \mathcal U &= \mathcal F^{-1} \mathcal F_u =
 \begin{pmatrix} 
 \ell+\frac{\omega_u}{2} & m & e^{\omega/2}\\
 - Qe^{-\omega} & \ell - \frac{\omega_u}{2} & 0\\
 0 & e^{\omega/2}& 0
 \end{pmatrix},  \label{eq:U}\\[0.2cm] 
 \mathcal V &= \mathcal F^{-1} \mathcal F_v=
 \begin{pmatrix} 
 m  - \frac{\omega_v}{2}
 & - R e^{-\omega}  & 0 \\
 \ell & m + \frac{\omega_v}{2} & e^{\omega/2}\\
 e^{\omega/2}& 0 & 0
\end{pmatrix}, 
 \label{eq:V}
\end{align}
 where 
\begin{equation}\label{eq:QandR}
 Q = \langle \f_{uuu}, \f\rangle, \quad 
R = \langle \f_{vvv}, \f\rangle, \quad 
\ell =  \langle \mathfrak H ,\f_u \rangle, \quad 
m =  \langle \mathfrak H ,\f_v \rangle, 
\end{equation}
 and $\mathfrak H$ is the mean curvature vector in \eqref{eq:Hexpression}.
 Moreover,  $\ell$, $m$, $Q$ and $R$ take purely imaginary values.
\end{Theorem}

\begin{proof}
 Writing
 $\mathcal F:= (\f_1, \f_2, \f_3)$ and  using the last equation in \eqref{eq:SU21},
 we obtain
 $\mathcal F^{-1} \d\mathcal F = \r{P} \bar{\mathcal F}^t \r{P} \d \mathcal F$.
 A straightforward computation now shows that
 \[
 \mathcal U =   \mathcal F^{-1} \mathcal F_u  = 
\left(
\begin{array}{rrr}
 \langle {\f_1}_u, \f_2\rangle &  \langle {\f_2}_u, \f_2 \rangle &  \langle
 {\f_3}_u , \f_2 \rangle \\
 \langle {\f_1}_u, \f_1\rangle &  \langle {\f_2}_u, \f_1\rangle &  \langle
 {\f_3}_u , \f_1 \rangle \\
 -\langle {\f_1}_u, \f_3\rangle &  -\langle {\f_2}_u, \f_3\rangle &  -\langle
 {\f_3}_u , \f_3 \rangle 
\end{array}
\right),
\]
 and $\mathcal V$ is obtained from $\mathcal U$ by switching the subscripts 
 $u$ and $v$.
 We want to compute the coefficients of $\mathcal U$ and $\mathcal V$ 
 in more detail. First we note that 
 by the definition of the coordinate frame $\mathcal F$
 in \eqref{eq:coordinateframe}, we have
\begin{align}
 {\f_1}_u &= e^{-\omega/2}\left( \f_{uu} - \frac{\omega_u}{2} \f_u\right), \quad
 {\f_2}_u = e^{-\omega/2} \left(\f_{vu} - \frac{\omega_u}{2} \f_v\right), 
 \quad {\f_3}_u = \f_u, \label{eq:fu}
\end{align}
 A straightforward computation by using 
 \eqref{eq:fu} shows that 
\begin{alignat*}{3}
\langle {\f_1}_u, \f_2\rangle &= e^{- \omega} \langle \f_{uu}, \f_v\rangle
 - \frac{\omega_u}{2}, \quad &
\langle {\f_2}_u, \f_2 \rangle &= e^{-\omega} \langle \f_{v u}, \f_v\rangle, 
 \quad & \langle {\f_3}_u , \f_2 \rangle &=e^{\omega/2}, \\
 \langle {\f_1}_u, \f_1\rangle &= e^{-\omega} \langle \f_{uu}, \f_u\rangle, 
 \quad & \langle {\f_2}_u, \f_1\rangle &= e^{-\omega} \langle \f_{v u}, \f_u\rangle
 - \frac{\omega_u}{2}, \quad&
 \langle {\f_3}_u , \f_1 \rangle &= 0, \\
 -\langle {\f_1}_u, \f_3\rangle&= - e^{- \omega/2}
 \langle \f_{uu}, \f\rangle, \quad &
 -\langle {\f_2}_u, \f_3\rangle &=  -e^{-\omega/2} \langle \f_{v u}, \f\rangle,
 \quad & -\langle {\f_3}_u , \f_3 \rangle  &= 0. 
\end{alignat*}
 Then it is easy to see that $ -\langle {\f_2}_u, \f_3\rangle = e^{\omega/2}$.
 Further, $\langle \f_{uu}, \f_v\rangle$ can be computed by \eqref{eq:Hexpression}
 as 
\begin{equation*}
 \langle \f_{uu}, \f_v \rangle = 
\langle \f_{u},\f_v \rangle_u - \langle \f_u, \f_{v u}\rangle
 = \omega_u e^{\omega}  - e^{\omega} \langle \f_u, \mathfrak H\rangle
\end{equation*}
 and thus $\langle {\f_1}_u, \f_2\rangle$ can be rephrased as
\begin{equation*}
\langle {\f_1}_u, \f_2\rangle = \frac{\omega_u}{2} - \langle \f_u, 
 \mathfrak H\rangle = \frac{\omega_u}{2} + \ell.
\end{equation*}
 Here, since $\mathfrak H$ is the mean curvature vector 
 $\f$, that is, $\mathfrak H$  can be represented by  
 $\{\mathfrak E_1, \mathfrak E_2\}=\{\sqrt{-1}e^{-\omega/2}\f_u, 
 \sqrt{-1}e^{-\omega/2}\f_v\}$, 
 $\Re \langle \mathfrak H, \f_u\rangle= 
 \Re \langle \mathfrak H, \f_v\rangle = 0$ and thus $- \langle \f_u, 
 \mathfrak H\rangle = \langle \mathfrak H, \f_u \rangle = \ell$, 
 that is, $\ell$ and $m$ take purely imaginary values. 
 Similarly,  we have 
 $\langle {\f_2}_u, \f_1\rangle = -\frac{\omega_u}{2} + \ell$.
 Further since $\langle \f_u, \f\rangle =0$ and
 $\langle \f_u, \f_u\rangle =0$, we have $ \langle \f_{uu}, \f\rangle =0$
 and thus $-\langle {\f_1}_u, \f_3\rangle = 0$.
 By using $\langle \f_u, \f_u \rangle =0$, 
 the second derivative of $\langle \f_u, \f\rangle =0$
 with respect to $u$  implies 
 $\langle \f_{uuu}, \f\rangle = -\langle \f_{uu}, \f_u\rangle$. 
 Moreover, the derivative of $\langle \f_{u}, \f_u\rangle= 0$ 
 with respect to $u$ implies that $\Re \langle \f_{uu}, \f_u\rangle
 = 0$, thus 
 $Q = \langle \f_{uuu}, \f\rangle$ takes purely imaginary values
  and $\langle {\f_1}_u, \f_1\rangle = - Q e^{-\omega}$.
 
 Finally we obtain $\mathcal U$ 
 as in \eqref{eq:U}.
 A similar computation for $\mathcal F^{-1} \mathcal F_v$ shows that 
 $\mathcal V$ is as in \eqref{eq:V}.
 \end{proof}
\begin{Definition}
 By using the purely imaginary functions $Q$, $R$, $\ell$ and $m$ in 
 \eqref{eq:QandR}, we 
 define two differentials  as
\begin{align}\label{eq:cubic}
 C & = Q \, \d u^3 + R \, \d v^3 = \langle \f_{uuu}, \f\rangle\, \d u^3 
 + \langle \f_{vvv}, \f\rangle \, \d v^3, \\
 L & = \ell \, \d u + m \,\d v = 
\langle \mathfrak H, \f_u \rangle\, \d u
 + \langle \mathfrak H, \f_{v}\rangle \, \d v. \label{eq:1form}
\end{align}
 The form $C$ will be called the {\it cubic differential} and 
 $L$ will be called the {\it mean curvature $1$-form}
 (some authors also call it ``Maslov form'').
 Both forms take purely imaginary values.
\end{Definition} 

\begin{Remark}
 The cubic differential $C$ and the mean curvature $1$-form $L$ 
 are defined by using a horizontal lift $\f$ instead of the original 
 immersion $f$, however, $C$ and $L$ are independent of the choice 
 of a horizontal lift. Thus they are an invariant of the timelike 
 Lagrangian immersion $f$.
\end{Remark}


\subsection{Fundamental theorem}
 The Maurer-Cartan form $\alpha = \mathcal F^{-1} \d \mathcal F$ of 
 the coordinate frame for a timelike Lagrangian immersion $f$ in $\CH$ satisfies 
 the Maurer-Cartan equation 
 $\d \alpha + \alpha \wedge \alpha = 0$,			     
 that is 
\[
 \mathcal U_{v} - \mathcal V_u+[\mathcal V, \mathcal U] =0
\]
 holds. Then a straightforward computation shows that 
 we have the following system of partial differential equations:
\begin{align}\label{eq:compatibility1}
&\omega_{u v}= e^{\omega}- Q R e^{-2 \omega} + m \ell, \\ 
&\ell_v - m_u =0, \label{eq:compatibility2}\\
& Q_v e^{-2 \omega} +(e^{-\omega} \ell)_u =0, 
 \quad 
 R_u e^{-2 \omega} +(e^{-\omega} m)_v=0.\label{eq:compatibility3}
\end{align}
 In the following we show the fundamental theorem of timelike 
 Lagrangian surfaces in $\CH$.
\begin{Theorem}\label{thm:fund}
 Let $f: \D \to \CH$ be a timelike Lagrangian immersion and 
 $\f$ a horizontal lift of $f$. Further let $\d s^2 = 2 e^{\omega}\d u \d v$,
 $C= Q\, \d u^3 + R\, \d v^3$ and $L = \ell\, \d u + m\, \d v$ be the metric, 
 the cubic differential and the mean curvature $1$-form of $f$. Then these 
 functions $\omega, Q, R, \ell$ and $m$ 
 satisfy the system of partial differential equations
 \eqref{eq:compatibility1}, \eqref{eq:compatibility2} 
 and \eqref{eq:compatibility3}.

 Conversely let $\d s^2 = 2 e^{\omega} \d u \d v$,  $C= Q\, \d u^3 + R\, \d v^3$ and 
 $L = \ell\, \d u + m\, \d v$ be defined by solutions of the system of partial 
 differential equations 
 \eqref{eq:compatibility1}, \eqref{eq:compatibility2} and \eqref{eq:compatibility3}
 with purely imaginary $Q$, $R$, $\ell$ and $m$. 
 Then there exists a timelike Lagrangian immersion $f$ such that 
 the metric, the cubic differential and the mean curvature $1$-form
 are $\d s^2$, $C$ and $L$, respectively.
\end{Theorem}
\begin{proof}
 We only need to prove the converse.
 Since $\omega$, $C$ and $L$ satisfy \eqref{eq:compatibility1}, 
 \eqref{eq:compatibility2} and \eqref{eq:compatibility3}, 
 there exists an $\mathcal F: \D \to \U$ such
  that $\mathcal F^{-1} \d \mathcal F 
 = \mathcal{U} \d u  + \mathcal{V} \d v$
 with $\mathcal{U}$ and $\mathcal{V}$ defined in \eqref{eq:U} and \eqref{eq:V}.
 Let $e_3=(0, 0, 1)^T$ and set $ \f= \mathcal F e_3$. Then it is easy to see 
 that $\f$ takes values in $H^5_{3} $ : $\langle \f, \f \rangle    = \langle \mathcal{F} e_3, \mathcal{F} e_3 \rangle  = - |a|^2,$ if $\mathcal{F} = a \mathcal{F}_0$, where the latter matrix has determinant $1$. But the determinant of $\mathcal{F}$ is in $S^1$, 
 whence $|a| = 1$. Moreover, it is also straightforward 
 to see 
\[
 \langle \f_u, \f_u\rangle =  \langle\f_v, \f_v\rangle=0, 
\]
 and  $\langle \f_u, \f_v\rangle$ takes values in $\R^{\times}$, that is,
 $\f$ is timelike, is parametrized by null coordinates and is Legendrian, 
 that is, $\Im \langle \f_u, \f_v \rangle =0$ holds. Finally taking
 the Boothby-Wang fibration $\pi: H^5_3 \to \CH$ for $\f$, that is $f = \pi\circ \f$, we have a timelike Lagrangian immersion $f$ in $\CH$.  
\end{proof}
\subsection{Minimality}
 In the following, we characterize minimality of 
 a timelike Lagrangian immersion in $\CH$ in terms of
 the invariant $1$-form $L$ defined in the previous section.
\begin{Proposition}\label{prp:minicharact}
 Let $f : \D \to \CH$ be a timelike Lagrangian immersion and 
 $L$  the differential $1$-form defined in \eqref{eq:1form}.
 Then $L$ is closed. Moreover, let $\mathcal{L}$ denote a purely imaginary integral of $L$ on $\D$.  Then the diagonal matrix 
 \begin{equation}\label{eq:D}
  D = \exp \left[\di \left(- \mathcal{L}, -  \mathcal{L}, \; 0 \right)\right] = 
 \di \left(\exp\left[- \mathcal{L},\right], 
 \exp \left[-\mathcal{L},\right], 1\right),
 \end{equation}
 is well-defined and 
 $\det (\mathcal FD)$ is constant. 
 Furthermore, $f$ is minimal if and only if $L\equiv 0$ if and only if
 the determinant $\det \mathcal F$ is constant.
 Thus, without loss of generality
  we can assume that $\det \mathcal F \equiv 1$ holds.
\end{Proposition}
\begin{proof}
 The closedness of $L$ follows from \eqref{eq:compatibility2}.
 Thus 
  \[
   D  = \di \left(\exp\left[- \mathcal{L} \right], 
 \exp\left[- \mathcal{L}\right], 1\right)
  \]
  is well-defined. It is also easy to see that 
  $\tr \{(\mathcal FD)^{-1} \d (\mathcal FD)\}=0$, 
  thus $\det (\mathcal FD)$ is constant.
  Since  $\mathcal{L}$ is only determined up to a purely imaginary  constant, 
  we can adjust this  constant such that the determinant of $\mathcal FD$ is identically $1$.
  
  Moreover,  by \eqref{eq:1form} the mean curvature vector $\mathfrak H$ 
  vanishes  if and only if $\ell = m =0$ and 
  this is equivalent with $L=0$.
\end{proof}
\begin{Remark}\label{rm:minimal}
\mbox{}
\begin{enumerate}
\item Since $L$ takes purely imaginary values, the function 
 $\exp\left(2 \mathcal L \right)$
 takes values $S^1$.
 It is called {\it  Lagrangian angle function}  of $f$, \cite[after Lemma 3.1]{ML}.
 
 \item Note that  $\mathcal{F}D$ is in $\U$.
 The last column  of this matrix is the same (horizontal lift) as the one of $\mathcal{F}$. However, the first two columns are rotated. Hence $\mathcal{F}D$ is in general no longer a coordinate frame of some timelike Lagrangian immersion.
 
 \item Since $\mathcal{F}D$ also has constant determinant (in $S^1$), one can easily change to a new ``frame'' which has determinant $1$.
 This can be done naturally in several  ways: one can multiply the  matrix $\mathcal{F}D$ by $a^{-1/3}$,
 where $a = \det (\mathcal{F}D)$ is 
 or one can multiply on the right by the matrix 
 $\hat{D} = \di (a^{-1/2}, a^{-1/2}, 1)$.
 Actually, by replacing the lift $\f$ of the original timelike Lagrangian 
 immersion $f$ by the 
 lift 
 $\hat{\f} = a^{-1} \f\; , a \in S^1,$ one obtains the same mean curvature $1$-form $L$ and thus automatically 
 $\hat{\mathcal{F}}D (0,0) = \id$, whence also  
 $\det(\mathcal{F}D) \equiv 1$.
 
 \item If a timelike Lagrangian surface $f$ is minimal, then of the two normalizations just discussed the last option seems to be preferable, since in this case the new immersion 
 $\hat{\f}$ leads directly to a coordinate frame which satisfies $\mathcal{F}(0,0) = \id$.

\item It is known that  there 
 are no compact minimal surfaces in a pseudo-Riemannian manifold 
 with non-positive curvature, see for example \cite[p. 379]{KNII}.
 \end{enumerate}
\end{Remark}


\section{Characterization of a timelike minimal Lagrangian surfaces}\label{sc:Chat}
 In this section,  we characterize a timelike minimal Lagrangian surface
 in terms of a family of flat connections. For this purpose, we first consider 
 the associated family of a timelike minimal Lagrangian surface and 
 naturally introduce the so-called ``spectral parameter'' $\lambda$ 
 into the Maurer-Cartan form $\alpha$ of the coordinate frame $\mathcal F$.
 In Theorem \ref{thm:Charact}, we characterize the minimality of 
 a timelike Lagrangian surface in terms of a family of connections $\d + \alpha^{\l}$.
 
\subsection{Associated family}
 Let $f: M \to \CH$ be a timelike minimal Lagrangian immersion.
 Then there exists the metric $2 e^{\omega} \d u \d v$,  
 the cubic differential $C$ and the mean curvature $1$-form $L$ 
 which vanishes identically on $M$ 
 associated to $f$. 
 Then it is clear  from \eqref{eq:compatibility1}, 
 \eqref{eq:compatibility2} and \eqref{eq:compatibility3},  
 that the integrability conditions for a minimal surface are $m = \ell = 0$ and the partial differential equation
 \begin{equation}\label{eq:Tzitzeica}
 \omega_{u v}= e^{\omega}- Q R e^{-2 \omega},\\ 
 \end{equation}
 with purely imaginary functions $Q$ and $R$ 
  satisfying $Q_v = 0$ and $R_u =0$, and 
 a real valued function $\omega$.
\begin{Remark}

 The equation \eqref{eq:Tzitzeica} is the original \textit{Tzitz\'eica
 equation} 
 of indefinite affine spheres up to sign which can be easily adjusted by 
 the change of coordinates $(u, v)$. However, the affine spheres have 
 \textit{real} cubic differential and 
 the timelike minimal Lagrangian surfaces in $\CH$ have 
 \textit{purely imaginary} cubic differential, thus a solution of 
 the Tzitz\'eica equation gives two different classes of surfaces.

\end{Remark}
 In this case, by Theorem \ref{thm:fund}, there exists a family of 
 solutions parametrized by $\lambda \in \R_{>0}$
 \begin{equation}\label{eq:hatphi}
  \left\{e^{\omega^{\lambda}},  C^{\lambda}, L^{\lambda}\right\}_{\lambda \in \R_{>0}}
 \end{equation}
such that 
\[
 \omega^{\lambda} = \omega, \quad 
 C^{\lambda} = \lambda^{-3} Q\> \d u^3 + \lambda^{3} R\> \d v^3, \quad 
 L^{\lambda} = L =0.
\]
 Then by Theorem \ref{thm:fund}, there exists a family of 
 timelike Lagrangian minimal surfaces 
 $\{{\hat f^{\lambda}}\}_{\lambda \in \R_{>0}}$ such that 
 $\hat f^{\lambda}|_{\lambda =1} = f$.
 It is natural to call the family $\{ {\hat f^{\lambda}}\}_{\lambda \in \R_{>0} }$ 
 the {\it associated family} of $f$. 
 The parameter $\l$ will be called the 
 \textit{spectral parameter}.

\begin{Remark}
 It is important to note that the parameter $\lambda$ above can actually be chosen from $\C^{\times}$ without restricting the integrability condition. 
 The solutions to some PDE's mentioned in the proof of Theorem \ref{thm:fund} can thus be computed for all $\lambda \in \C^{\times}$.
This is an important information, since in the discussion of the construction method via loop groups one will carry out the group splittings on the unit circle, while one discusses surfaces only for $\lambda \in \R_{>0}$.
\end{Remark}
 Let $\hat {\mathcal F}$ be the coordinate frame of a horizontal lift $\hat \f^{\l}$ 
 of $\hat f^{\l}$.
 Then the Maurer-Cartan form $\hat \alpha = \hat {\mathcal U} \d u + 
 \hat {\mathcal V} \d v$ of $\hat {\mathcal F}$
 for the associated family 
 $\{\hat f^{\l}\}_{\lambda \in \R_{>0}}$
 is given by $\hat {\mathcal U}$ and $\hat {\mathcal V}$ as in 
 \eqref{eq:U} and \eqref{eq:V} where we have replaced 
 $Q, R, \ell$ and $m$ by $\lambda^{-3} Q, \lambda^3 R$ , $0$ and $0$,
 respectively.
  
 Then consider 
\begin{equation}\label{eq:Gauge}
 F  =\hat {\mathcal F} G, \quad 
 G = \di ( \lambda, \lambda^{-1}, 1)
\end{equation}
 and thus 
\[
\alpha =  F^{-1}  \d F = U \d u + V \d v	   
\] 
 with $U = G^{-1}\hat{\mathcal U} G$ 
 and $V = G^{-1}\hat{\mathcal V} G$. Since $G$ takes values in 
 $\U$
 for any $\lambda \in \R_{>0}$, thus $G \hat{\mathcal F} e_3$ 
 is isometric to $\hat{\mathcal F} e_3$. Define 
 $f^{\l} = \pi \circ G \hat{\mathcal F} e_3$. Thus we do not distinguish 
 $\{\hat f^{\lambda}\}_{\lambda \in \R_{>0}}$ and 
 $\{f^{\lambda}\}_{\lambda \in \R_{>0}}$, and it will be also 
 called the associated family.
 
\subsection{A family of flat connections}\label{subsc:flatconnections}
 Let us return now  to the general case of a timelike Lagrangian immersion $f$, 
 with horizontal lift $\f$ and coordinate frame $\mathcal{F}$ with 
 $D$ defined in \eqref{eq:D}
 such that  $\mathcal{F} D (0,0) = \id$.
 Then it is easy to see that the Maurer-Cartan form $(\mathcal{F} D)^{-1} 
 \d (\mathcal{F} D) = \hat{\mathcal U} \d u + \hat{\mathcal V} \d v $ can be computed
 as
\begin{equation}\label{eq:FDMaurer}
 \hat{\mathcal U} = 
 \begin{pmatrix} 
 \frac{\omega_u}{2} & m & e^{\omega/2 + \mathcal L}\\
 - Qe^{-\omega} & - \frac{\omega_u}{2} & 0\\
 0 & e^{\omega/2- \mathcal L}& 0
 \end{pmatrix}, \quad   
 \hat{\mathcal V} = 
 \begin{pmatrix} 
  - \frac{\omega_v}{2}
 & - R e^{-\omega}  & 0 \\
 \ell &   \frac{\omega_v}{2} & e^{\omega/2 + \mathcal L}\\
 e^{\omega/2- \mathcal L}& 0 & 0
\end{pmatrix}.
\end{equation}

 From the discussion in the previous section, 
 it is natural to introduce a family of Maurer-Cartan forms $\alpha^{\l}$
 for the Maurer-Cartan form $\alpha$ of the timelike Lagrangian surface 
 $f: M \to \CH$  as 
  \begin{equation}\label{eq:alphalambda-orig}
 \alpha^{\l} = U^{\lambda} \d u + V^{\lambda} \d v,
 \end{equation} 
 for $\lambda \in \C^{\times}$, where 
 $U^{\lambda}$ and $V^{\lambda}$ are  given by  

\begin{equation}\label{eq:UVlambda0}
U^{\lambda} =
 \begin{pmatrix} 
\frac{\omega_u}{2} & \l m &
 \lambda^{-1} e^{\omega/2+ \mathcal L}\\
-\lambda^{-1} Qe^{-\omega} & - \frac{\omega_u}{2} & 0\\
0 & \lambda^{-1}e^{\omega/2- \mathcal L}& 0
\end{pmatrix},   \;
V^{\lambda} =
 \begin{pmatrix} 
 - \frac{\omega_v}{2}
 & - \lambda R e^{-\omega}  & 0 \\
 \l^{-1} \ell &\frac{\omega_v}{2} & \lambda e^{\omega/2+ \mathcal L}\\
\lambda e^{\omega/2- \mathcal L}& 0 & 0
\end{pmatrix}. 
\end{equation}

Note that in this general situation we permit, opposite to the last subsection, $m \neq 0$ and 
$\ell \neq 0$.

 It is clear that $\alpha^{\lambda}|_{\lambda=1}$ is 
 the Maurer-Cartan form of the frame $\mathcal F D$ of $f$. 
 In the following theorem using the family of Maurer-Cartan forms $\alpha^{\lambda}$, we 
 characterize, when  a timelike  Lagrangian surface in $\CH$ actually is minimal.
\begin{Theorem}\label{thm:Charact}
 Let $f : \D \to \CH$ be a timelike Lagrangian surface in $\CH$.
 Then the following statements are equivalent$:$
\begin{enumerate}
\item $f$ is minimal.
\item The mean curvature $1$-form $L = \ell \d u + m \d v$ vanishes.
\item $\d + \alpha^{\lambda}$ gives a family of flat connections
 on $\D \times \U$. 
\end{enumerate}
Moreover, if any of these three statements above holds, then we have $Q_v =0$ and $R_u =0$.
\end{Theorem}
\begin{proof}
 The equivalence $(1) \Leftrightarrow (2)$ follows from Proposition 
 \ref{prp:minicharact}.
 Let us compute the flatness of $\d + \alpha^{\lambda}$.
 In terms of $U^{\lambda}$ and $V^{\lambda}$, it is equivalent with
 $U^{\lambda}_v - V^{\lambda}_u + [V^{\lambda}, U^{\lambda}]=0$, and
 a straightforward computation shows that this is equivalent with 
 the following equations:
\begin{align*}
& \omega_{uv} - e^{\omega} + Q Re^{-2 \omega} -m \ell=0, 
 \\
& (m e^{-\omega})_v +  R_u e^{-2\omega} =0, \quad 
 (\ell e^{-\omega})_u + Q_v e^{-2\omega} =0, \\
& (\lambda^{-1} - \lambda^{2}) m=0, \quad (\lambda^{-2} - \lambda) \ell =0.
\end{align*}
 The first three equations are just \eqref{eq:compatibility1} and \eqref{eq:compatibility3}, respectively. 
 The remaining two  equations are satisfied 
 for  all $\lambda \in \C^{\times}$ if and only if 
 $m=\ell=0$ and this is equivalent with that 
 $f$ is minimal.
 This completes the proof.
\end{proof}
\begin{Remark}
 The choices of $U^{\l}$ and $V^{\l}$ in \eqref{eq:UVlambda} 
 are natural  in view of  the quasi $6$-symmetric  space (induced by the 
 order $6$ automorphism $\hat \sigma$) in Section \ref{sc:Gauss}. 
 The Maurer-Cartan form $\alpha$ can be 
 decomposed into the eigenspaces of the order $6$ automorphism $\hat \sigma$, 
 and the $j$-th degree of the spectral parameter $\l$ in the Maurer-Cartan form 
 corresponds to the $j$-th eigenspace. This will be explained in more detail in
 Section \ref{subsc:Cha}.
\end{Remark}
We have thus found by two different approaches to the same restricted matrices depending on $\lambda$ and it is clear that for the minimal  timelike Lagrangian case it suffices to consider matrices of the type
\begin{gather}
\label{eq:alphalambda}
 \alpha^{\l} = U^{\l} \d u + V^{\l} \d v, 
\intertext{with}
\label{eq:UVlambda}
U^{\lambda} =
 \begin{pmatrix} 
\frac{\omega_u}{2} & 0&
 \lambda^{-1} e^{\omega/2}\\
-\lambda^{-1} Qe^{- \omega} &  - \frac{\omega_u}{2} & 0\\
0 & \lambda^{-1}e^{\omega/2}& 0
\end{pmatrix},   \quad
V^{\lambda} =
 \begin{pmatrix} 
 - \frac{\omega_v}{2}
 & - \lambda R e^{-\omega}  & 0 \\
 0 & \frac{\omega_v}{2} & \lambda e^{\omega/2}\\
\lambda e^{\omega/2}& 0 & 0
\end{pmatrix}
\end{gather}
 and
 $Q_v = 0, R_u = 0$ and  $\omega_{uv} - e^{\omega} + Q Re^{-2 \omega}=0$.
\begin{Corollary}
 Let $Q$, $R$ be purely imaginary functions and $\omega$ be a
 real function. Moreover define $U^{\l}$ and $V^{\l}$ as in \eqref{eq:UVlambda} 
 and $\alpha^{\l} = U^{\l} \d u + V^{\l} \d v$, and assume $\d \alpha^{\l} +
 \alpha^{\l} \wedge \alpha^{\l} =0$ for $\l \in \C^{\times}$. Then 
 $Q_v = 0$, $R_u=0$ and 
 $\omega_{uv} - e^{\omega} + Q Re^{-2 \omega}=0$, and there 
 exists a family of timelike minimal Lagrangian surfaces $\{f^{\l}\}_{\l \in \R^{\times}}$
 in $\CH$ with the cubic differential $ C = \l^{-3} Q\, \d u +\l^3 R\, \d v$ and the induced metric 
 $\d s^2 = 2 e^{\omega} \d u \d v$. 
\end{Corollary}

\begin{Definition}
 The solution of $(F^{\lambda})^{-1} \d F^{\lambda} = \alpha^{\lambda}$ 
 defined in \eqref{eq:alphalambda} with  $U^\lambda$ and $V^\lambda$ 
 as in \eqref{eq:UVlambda} and with initial condition
 $F^{\lambda}(0, 0) = \id$ will be called the {\it extended frame}
 of a timelike minimal Lagrangian surface $f$. The associated 
 family $\{f^{\lambda}\}_{\lambda \in \R^{\times}}$ is defined 
 by $F^{\lambda} e_3|_{\lambda \in \R^{\times}}$ with $e_3 =(0, 0, 1)^T$.
\end{Definition}

\begin{Example}[Real projective space]
 Let $Q = R =0$, and  
 $\omega= 2 \log \left( \frac{2}{-2 + u v}\right)$.
 Then it is easy to see that $\omega$ is a
 solution of  the Tzitzeica equation 
 $\omega_{u v}= e^{\omega}- Q R e^{-2 \omega} = e^\omega$.
 Then the extended frame $F^{\l}$ can be explicitly obtained as
\[
  F^{\l} = \exp (\l^{-1} u N_+)  \begin{pmatrix} e^{-\omega/2} & 0 &0  
 \\0 & e^{\omega/2}& 0\\0 &0 &1 \end{pmatrix}\exp (\l v N_-) , \quad 
 N_+ = \begin{pmatrix} 0 & 0 & 1 \\ 0 & 0 & 0 \\
 0 & 1 & 0 \end{pmatrix}, \quad 
 N_- = - N_+^T.
\]
 Then the horizontal lift $\f = \f^{\l}|_{\l =1}$ 
 can be computed explicitly as 
\[
 \f =\frac{1}{2 -u v} \begin{pmatrix} 2u \\ 2 v \\ 2 + u v \end{pmatrix}.
\]
 Clearly $\f$ is the anti-de Sitter sphere $H^2_1$ in $\R^{3}_2$.
 And the immersion $f = \pi \circ \f$ is a part of the
 indefinite real projective space in $\CH$.
 \end{Example}

\begin{Example}[Clifford type cylinder]
 Let $Q = - R  = \sqrt{-1}$, then  $\omega = 0$ is a solution of 
 the Tzitzeica equation $\omega_{u v}= e^{\omega}- Q R e^{-2 \omega}$.
 Then the coefficient matrices of the Maurer-Cartan form of 
 $\alpha = U^{\lambda}_{vac} \d u + V^{\l}_{vac} \d v$ are constant and the equation 
 $(F^{\l})^{-1} \d F^{\l} = \alpha^\lambda$
  can be integrated directly. We obtain  $ F^{\l} = \exp(U^{\lambda}_{vac} u + 
 V^{\l}_{vac} v) $
 with
\[
 U^{\lambda}_{vac} = 
\l^{-1} 
\begin{pmatrix} 
0 & 0 & 1 \\ 
- \sqrt{-1} & 0 &0  \\
0 &1 &0 
\end{pmatrix}, \quad 
 V^{\lambda}_{vac} = \l 
\begin{pmatrix} 
0 & \sqrt{-1} & 0 \\ 
0 & 0 &1  \\
1 &0 &0 
\end{pmatrix}.
\]
 Clearly $F$ is the extended frame of some timelike minimal Lagrangian immersion and a
  direct computation shows that 
 the horizontal lift $\f = \f^{\l}|_{\l =1} = Fe_3|_{\l =1}$
 can be computed as 
\[
 \f = F_0 \f_0,
\]
where with $\delta = e^{ 2 \pi \sqrt{-1} /3}$ we have 
\[
\f_0 = 
 \frac{1}{\sqrt{3}}
\begin{pmatrix}
 e^{\sqrt{-1}( \delta u - \delta^2 v)}\\
 -e^{\sqrt{-1}( \delta^2 u - \delta v)} \\
e^{\sqrt{-1}( u - v)}
\end{pmatrix}, 
\quad \mbox{and}\quad
F_0=
\frac{1}{\sqrt{3}}
\begin{pmatrix} 
 - \sqrt{-1} \delta^2 & \sqrt{-1} \delta & -\sqrt{-1} \\
 \sqrt{-1} \delta & -\sqrt{-1} \delta^2 & \sqrt{-1} \\
 1  & -1  & 1
\end{pmatrix}.
\]
 Then an another direct computation shows that $F_0 \in \U$ 
 and $\langle \f_0, \f_0 \rangle=-1$. The timelike surface 
 $\f_0$ is an analogue 
 of the Clifford torus in $\mathbb {CP}^2$, see for example 
\cite{DM}. Let us consider the curves $v = -u+ a\; (a \in \R)$ , where $u$ and $v$ denote 
 on null coordinates.
 Then 
\[
 \f_0|_{v=-u+a} = \frac{1}{\sqrt{3}}( e^{-\sqrt{-1} u - \delta^2 a}, 
 -e^{-\sqrt{-1} u - \delta a}, e^{2 \sqrt{-1} u-a})^T. 
\]
 Therefore the surface close up, and 
 $\f_0$ becomes a cylinder.
 \end{Example}

\section{Legendrian lifts of general timelike Lagrangian immersions into $\CH$}
\label{sc:Lift}
 We have so far only considered timelike minimal Lagrangian immersions 
 from contractible open domains in $\R^2$ into $\CH$.
 In the literature, usually immersions are defined on arbitrary 
 Lorentz surfaces (or, more generally, on any  real surface of dimension 
 two).

 Thus, when considering a timelike Lagrangian immersion 
 $f:M \rightarrow \CH$, the question comes up whether there always 
 exists a Legendrian lift $\f: M \rightarrow H^5_3$.
 As a consequence, the question arises, in what sense, if any, a 
 timelike Legendrian immersion is naturally associated with 
 a given timelike Lagrangian immersion.

 At one hand, for any contractible open subset $U$ of $M$ such 
 a lift exists, by what was discussed in the beginning of this paper.
 So the main question is, in what sense a ``global'' lift exists.
 In this section we will show that either a given 
 timelike Lagrangian immersion $f:M \rightarrow \CH$ already 
 has a global 
 timelike Legendrian lift $\f:M \rightarrow H^5_3$, or there exists a threefold 
 cover $\hat{M}$ of $M$, such that the natural lift of $f$ to $\hat{f}: 
 \hat{M} : \rightarrow \CH$
 has a timelike Legendrian lift 
 $\hat{\f}:\hat{M} \rightarrow \CH$ of $\hat{f}$.
\begin{Definition}
 A Lagrangian map $f:M \rightarrow \CH$ is called \textit{liftable}, 
 if there exists some Legendrian map
 $\f: M \rightarrow H^5_3$ such that $f = \pi \circ \f$.
\end{Definition}
\subsection{The basic transformation formula for horizontal lifts and frames}
 Now let $M$ be a Lorentz surface and let  $f:M \rightarrow \CH$ 
 be a Lagrangian immersion. It is known,  see for example 
 \cite[Section 3.2]{TWeinstein}, that the universal cover of $M$ 
 is diffeomorphic to $\R^2$, where the Lorentz metric, however, 
 is not known, in general.
  For simplicity we will thus assume without loss of generality that the universal cover actually is equal to $\R^2$ and denote it by $\D$.

 By $ \tilde{\pi} : \D \rightarrow M$  we denote the universal covering 
 of $M$.  Then we infer that $\D$ is contractible and 
 $\tilde{f} : \D \rightarrow \CH, p \mapsto f \circ \tilde{\pi}(p),$ 
 is again a Lagrangian immersion.
 Moreover, by what was discussed previously,  $\tilde{f}$ admits a 
 global horizontal lift  $\tilde{\f}: \D \rightarrow H^5_3$. 
 Now it is easy to derive the following
 \begin{Proposition}
 We retain the notation and the assumptions just made above.
 Let $\pi_1(M)$ denote the fundamental group of $M$, considered as a group of  
 Deck transformations acting on $\D$.
  Then for $\gamma \in \pi_1(M)$ we obtain$:$
\begin{enumerate}
 \item $\gamma^*\tilde{\f}: \D \rightarrow H^5_3$ is another global horizontal 
  lift of $\tilde{f}$. 
  
\item  There exists some uniquely determined 
  scalar $\tilde{c}(\gamma) \in S^1$ satisfying 
  \begin{equation}
  \gamma^*\tilde{\f} = \tilde{c}(\gamma)\tilde{\f},
  \end{equation}
   
\item  The map
 $\tilde{c}: \pi_1 (M) \rightarrow S^1, \gamma \mapsto \tilde{c}(\gamma),$
  is a (well defined) homomorphism.
\end{enumerate}
\end{Proposition}
\begin{proof}
 Since  $\gamma \in \pi_1(M)$ acts on $\D$ as
\begin{equation}\label{eq:gamma}
 \gamma : (u, v) \to (x, y), \quad x = \gamma^x (u), \quad y =\gamma^y (v),
\end{equation}
 for some strictly increasing one variable functions $\gamma^x, \gamma^y$, 
 see  \cite{Kul},
 it follows 
 by a straightforward computation  that 
 the map $\gamma^*\tilde{\f}: \D \rightarrow H^5_3$ is another global horizontal 
 lift of $\tilde{f}$. Therefore $\gamma^*\tilde{\f} = \tilde{c}(\gamma)\tilde{\f}$ 
 with some uniquely determined  scalar $\tilde{c}(\gamma)$ in $S^1$. 
 By the uniqueness statement,  it follows that 
 $\tilde{c}: \pi_1(M) \rightarrow S^1$ is a (well defined) homomorphism.
\end{proof}
 Next we recall from \eqref{eq:coordinateframe} the definition of the 
 ``natural coordinate frame'' for a Lagrangian immersion and apply it to $\tilde \f$:
 \begin{equation}\label{eq:tcoordinateframe}
 \widetilde{\mathcal F} = \left( e^{-\omega/2}\tilde{\f}_u, 
 e^{-\omega/2} \tilde{\f}_v, \tilde{\f}\right),
\end{equation}
  where $(u, v)$ is a null coordinate system  on $\D$.
 We know that $\widetilde{\mathcal F}$ takes values in 
 $\U$.
 Then $|\det \widetilde{\mathcal{F}}|^2 = 1$ and 
 $\widetilde{\mathcal{F}} \in \U = S^1 \cdot \SU $ follows.
By \eqref{eq:gamma}
\[
 e^{\tilde{\omega}} \d u \d v = \gamma^*( e^{\tilde{\omega}} \d u \d  v) = 
(\gamma^x_u \gamma^y_v) e^{\gamma^*{\tilde{\omega} }}\d u \d  v
\]
and 
\[
 \tilde{c}(\gamma) \tilde{\f}_u (u, v) =  (\gamma^*\tilde{\f})_u (u, v)  
 = \tilde{\f}_u (\gamma^x(u), \gamma^y(v))  \cdot \gamma^x_u (u), 
\]
the frame defined by (\ref{eq:tcoordinateframe}) for $\gamma^*\tilde{\f}$ yields
\begin{equation} \label{transFF2}
\gamma^*\widetilde{\mathcal{F}} = 
\tilde{c}(\gamma) \widetilde{\mathcal{F}} k, \quad 
 \mbox{where} \quad k = \di \left( \sqrt{ (\gamma^x_u)^{-1} \gamma^y_v }, 
 \sqrt{\gamma^x_u (\gamma^y_v)^{-1} },1\right).
\end{equation}
\begin{Corollary}
 If $f$ is minimal Lagrangian, then the homomorphism 
 $\tilde{c}$ satisfies $\tilde{c}(\gamma)^3 = 1$ for all $\gamma \in \pi_1(M)$.
\end{Corollary}
\begin{proof}
Since $\det k = 1$ it suffices to note that under our assumptions the determinant of 
$\det \widetilde{\mathcal{F}}$ is constant.
\end{proof}

\begin{Theorem}
 Let $f: M \rightarrow \CH$ be a timelike minimal Lagrangian immersion 
 and let $\tilde{c}$ be the homomorphism defined above. 
 Moreover let $F$ be the family of frames defined in \eqref{eq:Gauge}.
 Then  we obtain for all $\gamma \in \pi_1 (M):$
\begin{align} \label{transFF}
\gamma^*\widetilde{\mathcal F} &= 
\tilde{c}(\gamma) \widetilde{\mathcal F} k, \quad  \mbox{where
$k =  \di \left( \sqrt{ (\gamma^x_u)^{-1} \gamma^y_v }, 
 \sqrt{\gamma^x_u (\gamma^y_v)^{-1} },1\right)$,}\\
\label{transframe}
\gamma^*F &= 
\tilde{c}(\gamma) F k, \quad \mbox{for $k$ as above.} 
\end{align}
 Moreover, the  homomorphism  $\tilde{c}$ 
 can be considered as a homomorphism into the group 
 consisting of three elements:
$\tilde{c}: \pi_1(M) \rightarrow \mathbb{X}_3$,  where $ \mathbb{X}_3 = \{ 1, \delta, 
 \delta^2\}$ with $\delta = e^{2 \pi \sqrt{-1}/3}.$
\end{Theorem}
\subsection{The main theorem about global horizontal lifts }
Using the results just obtained we are able now to clarify the relation between 
timelike minimal Lagrangian surfaces into $\CH$ and timelike minimal Legendrian surfaces into $H^5_3$.
\begin{Theorem}
 Let $M$ be a Lorentz surface and let  $f:M \rightarrow \CH$ 
 be a timelike minimal Lagrangian immersion.
Then either $f$ admits a global horizontal lift $\f: M \rightarrow H^5_3$ or otherwise  
 there exists a Lorentz surface $\hat{M}$ and a threefold covering 
 $\hat{\pi}_f :\hat{M} \rightarrow M$ 
 such that $\hat{f} = f \circ \hat{\pi}_f$ is a timelike minimal Lagrangian immersion  which admits a global horizontal lift $\hat{\f} : \hat{M} \rightarrow H^5_3$.
\end{Theorem}
\begin{proof}
Considering the homomorphism $\tilde{c}$  discussed in the last subsection we have only two possibilities:

Case 1: The homomorphism $\tilde{c}$ is trivial:
In this case we have $\gamma^*\widetilde{\mathcal F} = \widetilde{\mathcal F}$ 
 and $\widetilde{\mathcal F}$ descends to a horizontal lift
$\f:M \rightarrow H^5_3$  for $f$.

Case 2: The homomorphism $\tilde{c}$ is not trivial: hence the  image of $\tilde{c}$ is the 
group  $ \mathbb{X}_3 $.
 Let $\Gamma$ denote the kernel of $\tilde{c}$ and put $\hat{M} = \Gamma \backslash 
 \D$.
 Then $\widetilde{\mathcal F}$ descends to a horizontal map 
\[
 \hat{\f} : \hat{M}  \rightarrow  H^5_3, 
\]
  with  $\hat{M} = \Gamma \backslash \D$ and $\hat{\f}$ satisfies
\[
 \gamma^*\hat{\f} = \hat{c}(\gamma)\hat{\f},
\]
  where $\hat{c}: \pi_1(M)/\Gamma \rightarrow \mathbb{X}_3  \subset S^1$
 is the induced homomorphism.
  
 Clearly, $\Gamma$ is a normal subgroup of $\pi_1(M)$. Moreover, let $\xi$ denote some element of $\pi_1(M)$ satisfying $\tilde{c}(\xi) = \delta = e^{\frac{2 \pi \sqrt{-1}}{3}}$ and let  $\Xi$ denote the subgroup of $\pi_1(M)$ generated by $\xi$.
 Then the first isomorphism theorem for groups tells us 
\[
\pi_1(M) / \Gamma \cong \mathbb{X}_3, 
\] 
 and the second isomorphism theorem for groups tells us
\[
  \mathbb{X}_3 \cong \pi_1(M) / \Gamma \cong \Xi \Gamma /\Gamma \cong  \Xi /{\Xi \cap \Gamma}.
 \]
 As a consequence, the action of the group $\Xi$ on $\hat{M}$ 
 is realized by the group $\mathbb{X}_3 $.
 But the image of $\hat{c}$ is in $S^1$ and thus is annihilated by the 
 Boothby-Wang type projection.
 Thus the map
\begin{equation}
 \hat{f} : \hat{M} \rightarrow \CH, \quad \mbox{given by}\quad 
 \hat{f}  = \pi \circ \hat{\f}
\end{equation}
 is invariant under the action of $\pi_1(\hat M)$ and actually projects to $f$.
 The claim now follows from the following statements:

 $(1)$ $\hat{f} : \hat{M} \rightarrow \CH$ is a timelike minimal Lagrangian immersion 
 with global horizontal lift $\hat{\f},$
 
 $(2 )$ $\hat{M}$ is a threefold cover of $M$.
\end{proof}
\begin{Corollary} \label{threefoldcover}
 Let $M$ be any Lorentz surface and $f:M \rightarrow \CH$ a timelike minimal 
 Lagrangian immersion.
 Then either $f$ admits a global horizontal lift (which then is timelike minimal Legendrian) or there exists a threefold cover $\hat{\pi}: \hat{M} \rightarrow M$ such that $\hat{f} : \hat{M} \rightarrow 
 \CH$, given by  
 $\hat{f} =  \hat{\pi} \circ f$ has a global minimal  Legendrian lift to $H^5_3$.
 \end{Corollary}
 \begin{Corollary}
 If the Lorenz surface is contractible, then the notions of a timelike 
 minimal Lagrangian immersion 
 into $\CH$ and a timelike minimal Legendrian immersion from 
 $M$ into $H^5_3$ are equivalent.
\end{Corollary}
\begin{Remark}
 The theorem just above
 shows that the relation between minimal Lagrangian surfaces and timelike minimal Legendrian surfaces is a bit delicate.
 Each timelike minimal Lagrangian surface from $M$ to $\CH$  induces a timelike 
 minimal Legendrian surface in $H^5_3$ either on $M$ itself, or at least on some threefold cover $\hat{M}$. 
 
On the other hand,  a timelike Legendrian surface $f: M\rightarrow H^5_3$, a horizontal map from the Lorentz surface $M$ to $H^5_3,$ induces trivially  a timelike Lagrangian surface $f_0:M \rightarrow \CH$ by  projection to $\CH$ via the Boothby-Wang type fibration.

 The difficulty in the relation between these surface classes 
 is in the (in)coherence of their domains, as expressed in the theorem above.
\end{Remark}

\section{Real forms of $\LSL$}
 It is clear that the extended frame $F^{\lambda}$ introduced in the previous section
 takes values in the loop group of $\SU$.
 In this section, we show that the loop group 
 corresponding to a timelike minimal Lagrangian surface in $\CH$ 
 takes values in a particular real form of $\lsl$ (or more generally of
 the affine Kac-Moody Lie algebra of type $A_2^{(2)}$).
 
 
\subsection{Real forms of $\lsl$}
This subsection is  a brief digression which is intended to help to put this paper into a larger context.
\subsubsection{ The setting of this paper}
 A straightforward computation shows that 
 the Maurer-Cartan form $\alpha^{\lambda}$ in \eqref{eq:alphalambda}
 of the  extended frame $F^{\lambda}$ satisfies the following two 
 equations (where we write $\alpha (\lambda) = \alpha^\lambda$ 
 temporarily):
\[
 \hat \sigma (\alpha(\epsilon^{-1} \lambda)) = \alpha (\lambda), \quad 
 \hat \tau (\alpha( \bar \lambda)) = \alpha(\lambda),
\]
 where $\epsilon = e^{\pi \sqrt{-1}/3}$ is the sixth root of unity, 
 $\hat \sigma$ is an order  $6$
 linear outer automorphism of $\sl$ and 
 $\hat \tau$ is an anti-linear involution  of $\sl$ defined as follows$:$
\begin{gather}\label{eq:sita}
 \hat \sigma (X) = - \ad (\di(\epsilon^2, \epsilon^4, -1)  \r{P} )\> X^T\\
\intertext{and }
\label{eq:sita2}
 \hat \tau (X) = -\ad (\r{P})  \overline {X}^T,
\end{gather}
 where $\r{P}$ is defined in \eqref{eq:P}. 
  
 More precisely, the $\alpha$
 takes values in the following loop algebra:
\begin{equation}
 \lsl^{\tau} = \{ g : \C^{\times} \to 
 \sl\;|\; \sigma (g(\lambda)) = g(\lambda), \;\; 
\tau (g(\lambda)) = g(\lambda)\;\;\mbox{and $g$ is smooth}
 \},
\end{equation}
 where we defined 
 $\sigma (g) (\l) = \hat \sigma (g( \epsilon^{-1} \l))$ 
 and $\tau (g) (\l) = \hat \tau (g( \bar \l))$.
 Therefore, the extended frame $F$ takes values in the 
 loop group $\LSL^{\tau}$ whose Lie algebra is
 $\Lambda \sl_{\sigma}^{\tau}$:
\begin{equation}
 \LSL^{\tau} = \{ g : \C^{\times} \to 
 \SL\;|\; \sigma (g(\lambda)) = g(\lambda), \;\; 
\tau (g(\lambda)) = g(\lambda)\;\;\mbox{and $g$ is smooth}
 \},
\end{equation}
 where $\sigma$ is an order $6$ automorphism and $\tau$ is an anti-linear involution 
 defined by   
 $\sigma (g) (\l) = \hat \sigma (g( \epsilon^{-1} \l))$ 
 and $\tau (g) (\l) = \hat \tau (g( \bar \l))$ with 
\begin{align}\label{eq:sigma}
 \hat \sigma (g) &= \ad (\di(\epsilon^2, \epsilon^4, -1) \r{P} )\> (g^{T})^{-1}, \\
 \label{eq:tau}
 \hat\tau (g) &=\ad (\r{P})  \> (\overline{g}^{T})^{-1}.
\end{align}
 The order $6$ automorphism $\hat \sigma$ and the anti-linear involution naturally arise for 
 minimal Lagrangian surfaces as discussed in Section \ref{sc:Chat}.

 \subsubsection{The case of $A_2^{(2)}$}
 The present paper deals with the Lie group $\SL$ with 
 an outer automorphism $\hat \sigma$ 
 and some anti-linear involution $\hat \tau$.
 It is known \cite{Kac} that up to 
 isomorphisms the Lie algebra and the order $6$ 
 automorphism are uniquely determined, see 
 \cite[Section 7]{DFKW} for details.
 
 Therefore, in our discussion above we could only change
 the anti-linear involution $\tau$ on $\lsl$, 
 the so-called the {\it real form} involution.
 Thus we fix the order $6$ automorphism $\sigma$ and 
 discuss the classification of real form involutions.
 
 In fact we have up to inner isomorphisms the following classification of real forms of $\lsl$, 
 or more generally the affine Lie algebra of $A_2^{(2)}$.
\begin{Theorem}[\cite{HG}]\label{thm:HG}
 The real form involutions $\tau$ for $\lsl$
 are classified $($up to isomorphism$)$
 as follows$:$
\begin{align*}
&(1)\;\; \tau(g)(\lambda) = - \overline
 {g(1/\bar \lambda)}^T, &  
 (2)&\;\; \tau(g)(\lambda) =\ad (I_{2,1}\r{P})  \> \overline
 {g(1/\bar \lambda)}, \quad\\
 & (3)\;\; \tau (g)(\lambda) =  - \ad (I_{2,1})  \overline
 {g(1/\bar \lambda)}^T, & 
 (4)&\;\; \tau (g)(\lambda) = \overline
 {g(\bar \lambda)},\\
& (5)\;\;\tau (g)(\lambda) = -\ad (\r{P})  \overline
 {g(\bar \lambda)}^T, &&
\end{align*}
 where $I_{2, 1} = \di(1, 1, -1)$. The first three are called 
 the {\rm almost compact} types and the rest are called 
 the {\rm almost split} types.
\end{Theorem}

\begin{Corollary}
 The real form involution $(5)$ in Theorem $\ref{thm:HG}$ 
 is defined by the anti-linear involution $\hat \tau$ 
 defined in \eqref{eq:sita2}.
\end{Corollary}
\begin{Remark}
\mbox{}

\begin{itemize}
 \item[(a)] 
 Each involution in Theorem \ref{thm:HG} corresponds a particular 
 special class of surfaces:
\begin{enumerate}
\item[(1)] Minimal Lagrangian surfaces in $\mathbb {CP}^2$.
\item[(2)] Elliptic or hyperbolic affine spheres in $\mathbb {R}^3$.
\item[(3)] Minimal Lagrangian surfaces in $\mathbb {CH}^2$.
\item[(4)] Indefinite affine spheres in $\R^3$.
\item[(5)] Timelike minimal Lagrangian surfaces in $\CH$.
\end{enumerate}
\vspace{0.5cm}

\item[(b)] 
 For the first four cases listed above 
 a loop group procedure has already been developed which allows (at least in principle) to construct all the surfaces of the corresponding class.
 This is a consequence of the fact that these surfaces can be characterized by a certain ``Gauss map'' to be harmonic.
 Actually, a harmonic Gauss map has only been established explicitly 
 for minimal Lagrangian surfaces in $\mathbb{CH}^2$, that is, in the  
 case $(3)$. In all other cases the existence of a harmonic Gauss map can be concluded, since the Maurer Cartan form of the naturally associated moving frame admits the insertion of a parameter $\lambda \in S^1$ in such a way as it is know to correspond to a primitive harmonic map. Below we will modify the construction of $(3)$ in a generalized way
 so as to fit the purposes of this paper and to permit to prove a Ruh-Vilms type theorem.
 For the remaining cases a Gauss map will be constructed elsewhere.
\vspace{0.5cm}

\item[(c)]
 Actually, when trying to cover all surface classes falling under the scheme outlined above one also needs to consider what happens if one considers an anti-linear automorphism which is conjugated by an inner automorphism such that the induced anti-linear automorphism of the loop group/loop algebra still commutes with $\hat{\sigma}$, see 
 \cite[Section 7]{DFKW}.
As a matter of fact, such cases did already occur in the paper \cite{Ko;real} and will also occur at least in case $(2)$ above.
\end{itemize}

\end{Remark}

\section{Three quasi $6$-symmetric spaces and Gauss maps}\label{sc:Gauss}

 In this section we define for a timelike Lagrangian surface in $\CH$
 three quasi $6$-symmetric 
 spaces (see Definition \ref{def:k-symmetric}) as well as
 associated Gauss maps into these spaces.
 These Gauss maps are essentially the same.
 In Theorem \ref{thm:FinallCharact} we characterize 
 a timelike minimal Lagrangian surface in $\CH$ in terms of its Gauss maps,
 thus proving a Ruh-Vilms type  theorem.
 
\subsection{Various bundles}

 We first introduce three quasi $6$-symmetric spaces of dimension $7$
 which 
 are bundles over $H^5_3$. Our approach follows \cite{McIn} in spirit, 
 but, as a matter of fact, we translate
 the work of \cite{McIn} into an ``indefinite version'' of that paper.
 We consider altogether three spaces, $FL_1$, $FL_2$, and $FL_3$.
 We first choose a natural basis $e_1, e_2, e_3$ of $\C^3_2$. 

 $(1) \; FL_1:$ 
We now consider $\C^3_2$ as the real $6$-dimensional symplectic vector space 
 given by the symplectic form $\Omega = \Im \langle \;,\; \rangle$. 
 Then the family of oriented maximal Lagrangian ($\equiv$ isotropic) subspaces n
 of  $\C^3_2$  form a submanifold of the real Grassmannian $3$-spaces of $\C^3_2$, 
 that is, they form  the oriented Lagrangian Grassmannian 
 manifold $\LGr(3, \C^3_2)$. It is known \cite{Yuxin} that 
 $\LGr(3, \C^3_2)$ can be represented as a homogeneous space 
 $\U / \SOt$. If the Lagrangian Grassmannian $\LGr(3, \C^3_2)$
 is, in particular, an orbit of $\SU$,  it will be called 
 \textit{special} Lagrangian Grassmannian and it will be denoted by 
 $\SLGr(3, \C^3_2)$.
 
\begin{Proposition}\label{Prp:Grass}
  $\SU$ acts transitively on  $\SLGr(3, \C^3_2)$, 
  and we obtain
 \[
\SLGr(3, \C^3_2) = \SU / \SOt.
 \]
\end{Proposition}
Next we define: 
\begin{equation}
FL_1 = \{ (v,V)\;|\; v \in H^5_3, \; v \in V, \;
 V  \in  \SLGr(3, \C^3_2)\}.
\end{equation}
 It is easy to verify that $\SU$ acts on $FL_1$.
 Note that the natural projection from $FL_1$ to $\CH$ is a pseudo-Riemannian submersion which is equivariant under the natural group actions.
 Since $H^5_3 = \SU/\SUoneone$, where  $\SUoneone$ means 
 $\SUoneone \times \{1\}$, 
 the stabilizer at 
 \[
  (e_3, \operatorname{span}_{\R}\{ e_1, e_2,  e_3 \}) \in FL_1
 \]
  is clearly given by
 $\SUoneone \cap \SOt$, that is
 \[
  {\rm SO}_{1,1} =\{(a, a^{-1}, 1)\;|\; a \in \R^{\times}\}.
 \] 
 Therefore 
\[
 FL_1 = \SU/{\rm SO}_{1, 1}.
\]

 $(2)\; FL_2:$
 For the definition of $FL_2$, 
 we consider certain ``special regular complex flags'' in $\C^3_2$.
 Here by a regular complex flag we mean a sequence of four 
 complex subspaces, 
 $Q_0 = \{0\} \subset Q_1 \subset Q_2 \subset  Q_3 = \C^3_2$ 
 of $\C^3_2$, 
 where $Q_j$ has  complex dimension $j$.
 We then define the notion of a special regular complex flag 
 in $\C^3_2$ over $q \in H^5_3$  
 by requiring that we have a  regular complex flag in $\C^3_2$,  
 where the space $Q_1$  satisfies $Q_1 = \C q$.
 
 Thus we define:
 \begin{equation*}
FL_2 = 
 \{ (w,W)\;|\; w \in H^5_3, \; w \in W, \; \text{$W$ is a special regular complex 
 flag in $\C^3_2$}  
 \}.
\end{equation*}
  The definition of a special flag means that one can find three vectors, 
 $q_1, q_2, q_3 \in \C^3_2$ with
 $q = q_3$,  such that (using the signature of $\C^3_2$) the vectors $q_1$ and $q_2$ span a subspace with signature $(1,1)$. So we obtain a triple $q_1,q_2$ and $q_3$ as discussed in the previous case. By an argument analogous to the previous case we conclude 
 that  $\SU$  acts transitively on the family of special flags.
 Then the stabilizer of the action at a point 
 $(e_3, 0 \subset \C e_3 \subset  \C e_3 \oplus \C e_2
 \subset  \C e_3 \oplus \C e_2  \oplus \C e_1)$
 is again given by $\SOt \cap \di$, where $\di$ denotes the set of all
 diagonal matrices in $\SU$. Thus it is again $ {\rm SO}_{1,1}$.
\begin{Proposition}\label{Prp:Flags}
 $\SU$ acts transitively on $FL_2$, and it can be 
 represented as
\[
  FL_2 = \SU / {\rm SO}_{1, 1}.
\]
 \end{Proposition}
 Note that the natural projection from $FL_2$ to $\CH$ is a pseudo-Riemannian submersion which is equivariant under the natural group actions.

 $(3)\; FL_3:$ Finally, using the isometry group $\SU$ of $H^5_3$,
 we can directly define 
 a homogeneous space $FL_3$ as
\begin{equation}\label{eq:FL3}
FL_3 = \{ U P_1 \;U^T\;|\; \mbox{$U \in \SU$ and $P_1=\di (\epsilon^2, \epsilon^4, -1) \r{P}$}\},
\end{equation}
 where $\epsilon = e^{\pi \sqrt{-1}/3}$ and $\r{P}
 =\left(\begin{smallmatrix}0 & 1 & 0 \\ 1 & 0 & 0 \\ 0 & 0 & -1 \end{smallmatrix}\right)$ 
 as defined in \eqref{eq:P}.
 %
%
%
%
%
%
%
%
%
\begin{Theorem}
We  retain the assumptions and the notion above. 
 Then the following statements hold$:$
\begin{enumerate}
\item The spaces $FL_j, (j = 1,2,3)$ are homogeneous under the natural action of 
 $\SU$.
\item The homogeneous space $FL_j, (j = 1,2,3)$ can be represented as
 \[
       FL_j = \SU/ {\rm SO}_{1, 1}, 
 \quad \mbox{where}\quad 
{\rm SO}_{1, 1} = \{ \di (a,a^{-1}, 1)\;|\; 
 a\in \R^{\times} \}.
 \]
 In particular they are all $7$-dimensional.
\end{enumerate}
\end{Theorem}
\begin{proof}
$(1)$, $(2)$: The statements clearly follow from Proposition \ref{Prp:Grass}, 
 Proposition 
 \ref{Prp:Flags} and  the definition of $FL_3$ in \eqref{eq:FL3}
 and the stabilizer at $P_1$ is easily computed as 
  ${\rm SO}_{1, 1}$.
%
%
\end{proof}

 \subsection{Quasi $k$-symmetric spaces}
 It is easy to prove that 
 the fixed point set of 
 the real form involution $\hat \tau$ in \eqref{eq:tau} of $\SL$ is isomorphic to
 $\SU$, that is,
\[
 \SU \cong \{ g \in \SL\;|\; \hat \tau (g) = g\}.
\] 
 On the one hand, the order $6$ automorphism $\hat \sigma$ in \eqref{eq:sigma} acting 
 on $\SL$ does not naturally act on $\SU$, since $\hat \sigma$ and $\hat \tau$ 
 do not commute. However they have the following relation 
\begin{equation}\label{eq:sigmahat}
 \hat \sigma \hat \tau \hat \sigma = \hat \tau.
\end{equation}
 By abuse of notation we will also denote the order $6$ automorphism and the 
 real form involution on $\sl$ by $\hat \sigma$ and $\hat \tau$, respectively.
 Let $x_j \;, j \in \{0, 1, \dots, 5\}$, be  an eigenvector 
 of $\hat \sigma$ for the eigenvalue $\epsilon^j = e^{2 \pi\sqrt{-1} j/6}$, that is, $\hat \sigma x_j = \epsilon^j 
 x_j$.
 Then \eqref{eq:sigmahat} yields 
\[
\hat \sigma \hat \tau(x_j) 
 = \epsilon^j \hat \tau(x_j).
\]
 So $\hat \tau$ leaves invariant the eigenspaces 
 $\mathfrak g^{\C}_j\subset \sl$ of $\hat \sigma$.
 And the fixed point algebra of $\hat \tau$ is spanned by  all elements of the form
 $x_j + \hat \tau(x_j), (j = 0,1,\dots,5)$ and $x_j$ arbitrary in $\mathfrak g^{\C}_j$.
 So the real form decomposes according to the eigenspaces of $\hat \sigma$ 
 and   we have 
 \begin{equation}
  \hat \sigma (x_j + \hat \tau(x_j)) 
 = \epsilon^j (x_j + \hat \tau(x_j)).  
 \end{equation}
 Thus  $\hat \sigma$ is not an automorphism of the real 
 Lie algebra $\textrm{Fix}(\hat \tau)$, but its action on 
 $\textrm{Fix}(\hat \tau)$ is easy to describe, see Section \ref{subsc:Cha}.
\begin{Definition}\label{def:k-symmetric}
 Let $G/K$ be a real homogeneous  space, and $\hat \tau$
 a real form involution  acting on the 
 complexification $G^{\C}$ of $G$, such that  $G = \textrm{Fix} (\hat \tau)$.
 Moreover,  let $\hat \sigma$ 
 be an order $k \;(k\geq 2)$ automorphism acting on $G^{\C}$.
 Then $G/K$ will be called a \textit{quasi $k$-symmetric space} if 
 the following conditions are satisfied
\begin{enumerate}
 \item 
 $\hat \sigma \hat \tau \hat \sigma = \hat \tau$, 
\item $\textrm{Fix}(\hat \sigma)^\circ 
 \subset K^{\C} \subset \textrm{Fix}(\hat \sigma)$.
\end{enumerate}
\end{Definition}
\begin{Remark}
  If the automorphism $\hat \sigma$ of $G^{\C}$ in fact acts on $G$, 
 then the \textit{quasi} $k$-symmetric 
  space $G/K$ is just a $k$-symmetric space. 
 However, by
 condition $(1)$, this happens if and only if $k=2$, and thus a quasi
 $k$-symmetric space $G/K$ is a $k$-symmetric space if and only if 
 it is a ($2$-)symmetric space.
\end{Remark}
\begin{Corollary}
 The homogeneous spaces $FL_j \;(j=1, 2, 3)$ are quasi $6$-symmetric spaces.
\end{Corollary}
\begin{proof}
 First we note that the group $G = \SU$ has the
complexification $G^\C = \rm{SL}_3 \C$ and is the fixed point group of the real form involution 
$\hat{\tau}$ given in \eqref{eq:sita2}.

 We show that $FL_3$ is a quasi $6$-symmetric space. 
 First note that the stabilizer 
  \[
 \textrm{Stab}_{P_1} = \{ X \in \SU \;|\; X P_1 \; X^T = P_1\}. 
 \]
 at the point $P_1=\di (\epsilon^2, \epsilon^4, -1) \r{P}$ of $FL_3$
 is ${\rm SO}_{1, 1}$.
It is easy to verify that the order $6$-automorphism $\hat \sigma$ of $\rm{SL}_3 \C$ given
 in \eqref{eq:sigma} and the real form involution $\hat \tau$
 in \eqref{eq:tau} satisfy the condition $(1)$ in Definition  \ref{def:k-symmetric}. 
 Moreover,  a direct computation shows that the fixed point of 
 $\hat \sigma$ in $\rm{SL}_3 \C$ is ${\rm SO}_{1, 1}^\C$. 
 Thus $\textrm{Stab}_{P_1}$ satisfies 
 the condition $(2)$ in Definition  \ref{def:k-symmetric}. 
 
 Thus   $FL_3$ is quasi $6$-symmetric space  in the sense of Definition 
 \ref{def:k-symmetric}. 
 Furthermore, for any pair of homogeneous spaces 
 $FL_j$ and $FL_m$ 
 there exists a diffeomorphism 
 \[
  \phi_{jm}: FL_m \rightarrow FL_j
 \] 
 and a homomorphism
 $\chi_{jm} : 
 \SU \to \SU$
 such that for any $g \in \SU$ and $p\in FL_m$  we have 
\[
 \phi_{jm} (g.p) = \chi_{jm}(g).\phi_{jm} (p),
\]
 As a consequence, also $FL_1$ and $FL_2$ are  $6$-symmetric spaces.
\end{proof}

\subsection{Normalized Gauss maps of timelike Lagrangian surfaces in $\CH$}
 We now define three Gauss maps for a timelike Lagrangian surface $f$ in $\CH$. 
 Let us  assume that $f$ is defined on a simply connected domain $M$ and that
 $\f$ is a horizontal lift of $f$. Then we define the coordinate frame 
 $\mathcal F: M \to  \U$ as in \eqref{eq:coordinateframe}. 
Moreover, take the diagonal matrix $D$ as in \eqref{eq:D} and 
 consider the normalized coordinate frame $\mathcal F D$. If necessary 
 replacing $\f$ by $a \f$ with some constant $a \in S^1$, without loss
 of generality, we can assume 
\begin{equation}\label{eq:normalizedframe}
 \hat F= \mathcal F D : M \to \SU,
 \end{equation}
 and $\hat F$ will be called the \textit{normalized frame}.
 Note that 
\[
D = \di (\exp (-\mathcal L), \exp (-\mathcal L), 1)
\]
 where $\mathcal L$ is the integral of the mean curvature $1$-form $L$
 as in \eqref{eq:1form}. Thus $D$ is well-defined on $M$.
 Furthermore, the normalized frame is well-defined up to ${\rm SO}_{1, 1}$, that is, any normalized frame is of the form
\begin{equation}\label{Fk}
 \hat F k, \quad k \in {\rm SO}_{1, 1},
\end{equation}
 since by the freedom of the null coordinates $(u, v)$ by $(s(u), t(v))$, 
 where $s, t$ are positive functions of one variable each.

\begin{Definition}\label{dfn:normalizedGauss}
 Consider the projections  
 $\pi_j \circ \hat F : M \to FL_j (j =1, 2, 3)$, where $\pi_j: \SU \to FL_j$.
 Then 
 \[
  g_j = \pi_j \circ \hat F \quad (j =1, 2, 3)
 \]
 will be called the \textit{normalized Gauss maps} of $f$ (with values in $FL_j$). Note that by  \eqref{Fk} the maps $g_j$ are well-defined on $M$, that is,
  independent of coordinates.
\end{Definition}
 Our definitions were a priori not very geometric.
 But by following \cite{McIn} we find analogously
 three obvious geometric
 interpretations of the Gauss map.

 For $FL_1$: Let $g_1 : M \to FL_1$ be given by 
\[
 p \mapsto (\f(p), \>
 \operatorname{span}_{\R} \{ (e^{- \mathcal L} \f_{u})(p), 
 (e^{- \mathcal L} \f_{v})(p), \f(p) \}),
\]
 where  $\f$ is a  horizontal lift of $f$ such that $\det \mathcal F D 
 = 1$.

 For $FL_2$: Let $g_2: M \to FL_2$ be given  by 
\[
p \mapsto  
 (\f(p), 0 \subset \C \f(p) \subset  \C \f(p) \oplus \C \f_u(p) 
 \subset  \C \f(p) \oplus \C \f_u(p)  \oplus \C \f_{v} (p)).
\]
 On the other hand we can represent the Gauss map $g_3$ by using 
 the frame $\hat F$ defined in \eqref{eq:normalizedframe} as 
\[
 g_3 = \hat F  P_1 \>\hat F^T, \quad \mbox{with} \quad P_1=\di (\epsilon^2, \epsilon^4, -1) \r{P},
\]
 where $\epsilon = e^{\pi \sqrt{-1}/3}$ and $\r{P}$ is defined in \eqref{eq:P}.
\begin{Remark}
\mbox{}
\begin{enumerate}
 \item 
 From the above arguments, it is clear that 
 the normalized Gauss maps $g_j$ do not depend on 
 choices of the coordinates $(u, v)$, but only depend on the Lorentz structure of $M$.
 \item 
 It is known that the \textit{natural} 
 Gauss map $\tilde{g}_1 $ of a timelike Legendre immersion $\f: M \to H^5_3$ is
 given by the wedge product of $\f$
  with the tangential Gauss map 
 $\gamma : M \to \Gr_{\R} (2, \C^3_2)$, that is 
 $\tilde g_1 =(\f, \gamma\wedge \f)$, and that
 $\tilde g_1$ takes values in $FL_1$ if and only if $\f$ is 
 minimal  (see \cite{McIn} for example for the $\mathbb {CP}^2$ 
 case). Our \textit{normalized} Gauss map $g_1$ takes values 
 in $FL_1$ but it does not imply 
 minimality of a timelike Legendre immersion $\f$ since 
 we rotate the tangential Gauss map $\gamma$ by the 
 factor $e^{-\mathcal L} \in S^1$, and thus obtain $g_1 = (\f, (e^{-\mathcal L}\gamma)\wedge \f).$
\end{enumerate}

\end{Remark}
 Let $M$ be a two-dimensional Lorentz surface with null coordinates
 $(u, v) \in \mathbb D \subset M$, and let $G/K$ be a quasi
 $k$-symmetric space, $k>2$ and consider a smooth 
 map $g : M \to G/K$. Moreover, let $F: \mathbb D
 \to G$ be a frame of $g$ and $\alpha = F^{-1} \d F$ be the Maurer-Cartan
 form of $F$. According to the decomposition of $\mathfrak g = 
 \mathfrak k + \mathfrak p$ of $G/K$ where $\mathfrak k$ is 
 the Lie algebra of $K$, we have 
\[
 \alpha = \alpha_{\mathfrak k} + \alpha_{\mathfrak p}=  
 \alpha_{\mathfrak k} + \alpha_{\mathfrak p}^{u} +  
 \alpha_{\mathfrak p}^{v},
\]
 where the superscripts $u$ and $v$ denote the $u$- and $v$-parts, 
 respectively. Let us denote by $\hat \sigma$ also the differential of 
 the order $k$-automorphism, that is, 
 $\hat \sigma$ is the order $k$-automorphism of the Lie algebra 
 $\mathfrak 
 g^{\C}$ of $G^{\C}$. Then it is easy to see that $\hat \sigma$ has the eigenvalues 
 $\{\epsilon^0, \epsilon^1, \dots, \epsilon^{k-1}\}$ with $\epsilon = e^{2 \pi \sqrt{-1}/k}$
 and 
 the complexification $\mathfrak 
 g^{\C}$ can be decomposed into $k$-eigenspaces as
\[
 \mathfrak  g^{\C} =  \mathfrak  g^{\C}_0 + \mathfrak  g^{\C}_1+ \cdots +
 \mathfrak  g^{\C}_{k-1}.
\]
 Here $\mathfrak g_j^{\C} = 
 \{X \in \mathfrak g^{\C}  \;|\; \hat \sigma (X) =\epsilon^{j} X\}$.
 Note that $\mathfrak g_0^{\C} = \mathfrak k^{\C}$.
\begin{Definition}
 We retain the notation as above. 
 A smooth map $g: M \to G/K$ is called a \textit{Lorentz primitive 
 harmonic map} if the following conditions hold$:$
\begin{equation}\label{eq:Lprimitive}
 \mbox{$\alpha_{\mathfrak p}^{u}$ and $\alpha_{\mathfrak p}^{v}$
   take values in $\mathfrak  g^{\C}_{k-1}$ and $\mathfrak  g^{\C}_{1}$, 
 respectively}.
\end{equation}
\end{Definition}


\subsection{Characterization of the minimality in terms of the normalized 
 Gauss maps}\label{subsc:Cha}
 It is easy to compute the Maurer-Cartan form of the normalized frame $\hat F$
 as in \eqref{eq:normalizedframe}, see \eqref{eq:FDMaurer}:
\[
 \hat \alpha = \hat F^{-1} \d \hat F = \hat U \d u + \hat V \d v,
\]
 with
\begin{equation}\label{eq:hatUV}
\hat U =
 \begin{pmatrix} 
\frac{\omega_u}{2} & m &
 e^{\omega/2+ \mathcal L}\\
- Qe^{-\omega} & - \frac{\omega_u}{2} & 0\\
0 & e^{\omega/2-\mathcal L}& 0
\end{pmatrix},   \quad
\hat V =
 \begin{pmatrix} 
- \frac{\omega_v}{2}
 & - R e^{-\omega}  & 0 \\
 \ell &  \frac{\omega_v}{2} & e^{\omega/2+\mathcal L}\\
e^{\omega/2- \mathcal L}& 0 & 0
\end{pmatrix}. 
\end{equation}
 From Section \ref{subsc:flatconnections}, is natural to introduce the spectral parameter 
 into the Maurer-Cartan form as follows:
\begin{equation}\label{eq:hatUVlambda}
 \hat \alpha^{\l} = \alpha^{\lambda}
\end{equation}
 where $\alpha^{\lambda}$ is defined in \eqref{eq:alphalambda-orig}.

 On the other hand a straightforward computation shows that the eigenspaces $\mathfrak g_j^{\C} \subset 
 \sl$
 of the order $6$-automorphism $\hat \sigma$ in \eqref{eq:sita} are
\begin{align*}
 \mathfrak g_0^{\C} &= \left\{ \begin{pmatrix} a_{11} & 0 & 0 \\ 0 & - a_{11} &0  \\ 0 & 0 & 0 \end{pmatrix}\;\Bigg|\; a_{11} \in \C \right\}, \quad 
& \mathfrak g_1^{\C} &= \left\{ \begin{pmatrix} 0 & a_{12} & 0 \\ 0  & 0  &a_{23}  \\ a_{23}
 & 0 & 0 \end{pmatrix}\;\Bigg|\; a_{12}, a_{23} \in  \C \right\},
  \\
 \mathfrak g_2^{\C} &=  \left\{ \begin{pmatrix} 0 & 0 & a_{13} \\ 0 & 0 &0  \\ 0 & -
 a_{13} & 0 \end{pmatrix}\;\Bigg|\; a_{13} \in \C
 \right\},\quad 
 & \mathfrak g_3^{\C} &= \left\{ \begin{pmatrix} a_{11} & 0 & 0 \\ 0 & a_{11} &0  \\ 0 & 0 & -2 a_{11} \end{pmatrix}\;\Bigg|\; a_{11} \in \C \right\}, \\
 \mathfrak g_4^{\C} &= \left\{ \begin{pmatrix} 0 & 0 & 0 \\ 0 & 0 & a_{23}  \\ -a_{23} & 0 & 0 \end{pmatrix}\;\Bigg|\; a_{23} \in \C  \right\},
 \quad 
 &\mathfrak g_5^{\C} &= 
\left\{\begin{pmatrix} 0  & 0  & a_{13} \\ a_{21} & 0 & 0 \\ 
 0 & a_{13} & 0 \end{pmatrix}\;\Bigg|\; a_{21}, a_{13} \in \C \right\}.
\end{align*}
 We now have the main theorem of this paper.
\begin{Theorem}\label{thm:FinallCharact}
 Let $f: M \to \CH$ be a timelike Lagrangian surface in $\CH$.
 Then the following statements are equivalent$:$
\begin{enumerate}
\item $f$ is minimal.
\item The mean curvature $1$-form $L = \ell \, \d u + m \, \d v$ vanishes.
\item $\d + \hat \alpha^{\lambda}$ gives a family of flat connections
 on $\D \times \SU$.
\item The normalized Gauss maps $g_j\;(j =1, 2, 3)$ are respectively 
 Lorentz primitive harmonic maps  into the quasi $6$-symmetric spaces
 $FL_j\;(j =1, 2, 3)$.
\end{enumerate}
\end{Theorem}

\begin{proof}
 The equivalences of $(1)$, $(2)$ and $(3)$ follow  from Theorem \ref{thm:Charact}.
 We now show  the equivalence of $(4)$ and $(1)$. First note that 
 the normalized frame 
 $\hat F$ is common for all the normalized Gauss maps $g_1, g_2$ and $g_3$.
 
 From the eigenspace decomposition of $\hat \sigma$,
 $\alpha_{\mathfrak p}^u$ and $\alpha_{\mathfrak p}^v$ can be 
 computed as 
\[
 \alpha_{\mathfrak p}^u = 
 \begin{pmatrix} 
0 & m & e^{\omega/2+\mathcal L}\\
-Qe^{-\omega} & 0 & 0\\
0 & e^{\omega/2-\mathcal L}&0
\end{pmatrix}\, \d u ,   \quad 
 \alpha_{\mathfrak p}^v = 
 \begin{pmatrix} 
 0 & - R e^{-\omega}  & 0 \\
 \ell & 0 & e^{\omega/2+ \mathcal L}\\
e^{\omega/2-\mathcal L}& 0 & 0
\end{pmatrix}\, \d v. 
\]
 Thus it is easy to see that $ \alpha_{\mathfrak p}^u$ and
 $ \alpha_{\mathfrak p}^v$ respectively take values in 
  $\mathfrak g_5^{\C}$ and $\mathfrak g_1^{\C}$ if and only if 
 $L = m \, \d u + \ell\, \d v = 0$.
 Therefore by \eqref{eq:Lprimitive}, $g_j (j =1, 2, 3)$ is a 
 Lorentz primitive harmonic map into $FL_j (j =1, 2, 3)$
 if and only if $f$ is minimal.
\end{proof}

\textbf{Acknowledgement:} We would like to thank Hui Ma for pointing 
 out to us  (5) in Remark \ref{rm:minimal}. We would also
 like to thank the referee for carefully reading the manuscript and for pointing out to 
 us a number of typographical errors in the manuscript.

\end{document}